\documentclass[12pt]{amsart}
\usepackage{amssymb,amscd,amsfonts, slashed}
\usepackage{amsmath,amsfonts,enumerate,epsfig,epstopdf,url,array,tikz,tikz-cd}
\usepackage{mathrsfs,mathtools}
\usepackage[arrow,matrix,curve]{xy}
\usepackage{fullpage}
\usepackage{graphicx}
\usepackage{slashed}

\newcommand\RE{\mathbb{R}}
\newcommand\R{\mathbb{R}}

\newcommand\ZA{\mathbb{Z}}
\newcommand\Z{\mathbb{Z}}
\newcommand\ZZ{\mathbb{Z}}
\newcommand\QQ{\mathbb{Q}}

\newcommand\C{\mathbb{C}}

\newcommand{\beq}{\begin{equation}}
\newcommand{\eeq}{\end{equation}}
\usepackage{color}
\definecolor{darkgreen}{cmyk}{1,0,1,.2}
\definecolor{m}{rgb}{1,0.1,1}
\definecolor{green}{cmyk}{1,0,1,0}

\definecolor{test}{rgb}{1,0,0}
\definecolor{cmyk}{cmyk}{0,1,1,0}

\newcommand\Tr{\operatorname{Tr}}

\newcommand\Ind{\operatorname{Ind}}

\newcommand\rank{\operatorname{rank}}
\newcommand\tr{\operatorname{tr}}

\newcommand\ch{\operatorname{ch}}

\newcommand\fR{\mathfrak R}
\newcommand\fM{\mathfrak M}

\newcommand\cP{\mathcal P}

\newcommand\ep{\mathrm{ep}}
\newcommand\eps{\mathrm{ep,spin}}
\newcommand\spin{\mathrm{spin}}

\theoremstyle{plain}
\newtheorem{theorem}{Theorem}[section]
\newtheorem{lemma}[theorem]{Lemma}
\newtheorem{proposition}[theorem]{Proposition}
\newtheorem{corollary}[theorem]{Corollary}

\theoremstyle{definition}
\newtheorem{definition}[theorem]{Definition}
\newtheorem{definition*}{Definition}
\newtheorem{example}[theorem]{Example}
\theoremstyle{remark}
\newtheorem{remark}[theorem]{Remark}
\newtheorem{remarks}[theorem]{Remarks}
\newtheorem{remarks*}{Remarks}

\begin{document}

\title[Positive scalar curvature via end-periodic manifolds ]
{Positive scalar curvature metrics\\ via end-periodic manifolds}

\author{Michael Hallam}
\address{Mathematical Institute, 
Andrew Wiles Building, University of Oxford,
Oxford OX2 6GG, UK}
\email{michael.hallam@maths.ox.ac.uk}

\author{Varghese Mathai}
\address{Department of Mathematics, University of Adelaide,
Adelaide 5005, Australia}
\email{mathai.varghese@adelaide.edu.au}

\begin{abstract}
We obtain two types of results on positive scalar curvature metrics for compact spin manifolds that are even dimensional.
The first type of result are obstructions to the existence of positive scalar curvature metrics on such manifolds, expressed in terms 
of end-periodic eta invariants that were defined by Mrowka-Ruberman-Saveliev \cite{MRS}.  These results are the even dimensional
analogs of the results by Higson-Roe \cite{HigsonRoe}. The second type of result studies the number of path components
of the space of positive scalar curvature metrics modulo diffeomorphism for compact spin manifolds that are even dimensional, whenever this space
is non-empty. These extend and refine certain results in Botvinnik-Gilkey \cite{BG} and also \cite{MRS}. 
End-periodic analogs of \(K\)-homology and bordism theory are defined and are utilised to prove many of our results.
 \end{abstract}

\keywords{positive scalar curvature metrics, maximal Baum-Connes conjecture, 
end-periodic manifolds, end-periodic eta invariant, vanishing theorems, end-periodic \(K\)-homology, end-periodic bordism}

\subjclass[2010]{Primary 58J28; Secondary 19K33, 19K56, 53C21.}

\date{}
\maketitle

\tableofcontents

%%%%%%%%%%%%%
\section{Introduction}
%%%%%%%%%%%%%

Eta invariants were originally introduced by Atiyah, Patodi and Singer \cite{APS1, APS2, APS3} as a correction term appearing in an index theorem for manifolds with odd-dimensional boundary. The eta invariant itself is a rather sensitive object, being defined in terms of the spectrum of a Dirac operator. However, when one considers the \emph{relative} eta invariant (or \emph{rho invariant}), defined by twisting the Dirac operator by a pair of flat vector bundles and subtracting the resulting eta invariants, many marvellous invariance properties emerge. For example, Atiyah, Patodi and Singer showed that the mod \(\Z\) reduction of the relative eta invariant of the signature operator is in fact independent of the choice of Riemannian metric on the manifold. Key to the approach is their index theorem for even dimensional manifolds with global 
boundary conditions, which they show is equivalent to studying manifolds with cylindrical ends and imposing (weighted) $L^2$ decay conditions.

The links between eta invariants and metrics of positive scalar curvature metrics
have been studied using different approaches by Mathai \cite{M92,M92-2}, Keswani \cite{Keswani99} and Weinberger \cite{Weinberger}. 
A conceptual proof of the approach by Keswani, was achieved in the paper by Higson-Roe \cite{HigsonRoe} using 
\(K\)-homology; see also the recent papers by Deeley-Goffeng \cite{DG}, Benameur-Mathai \cite{BM, BM-JGP, BM-JNCG}
and Piazza-Schick \cite{PiazzaSchick07,PiazzaSchick07-2}.

Our goal in this paper to use the results of 
Mrowka-Ruberman-Saveliev \cite{MRS} instead of those by Atiyah-Patodi-Singer \cite{APS1}. Manifolds with cylindrical ends studied in \cite{APS1} are special cases of 
end-periodic manifolds studied in \cite{MRS}. More precisely, let $Z$ be a compact manifold with boundary $Y$
and suppose that $Y$ is a connected submanifold of a compact oriented manifold $X$ that is Poincar\'e dual to a primitive cohomology class $\gamma \in H^1(X, \ZA)$.
Let $W$ be the fundamental segment obtained by cutting $X$ open along $Y$ (Figure \ref{fig:epc}).

\begin{figure}[h]
\includegraphics[height=1in]{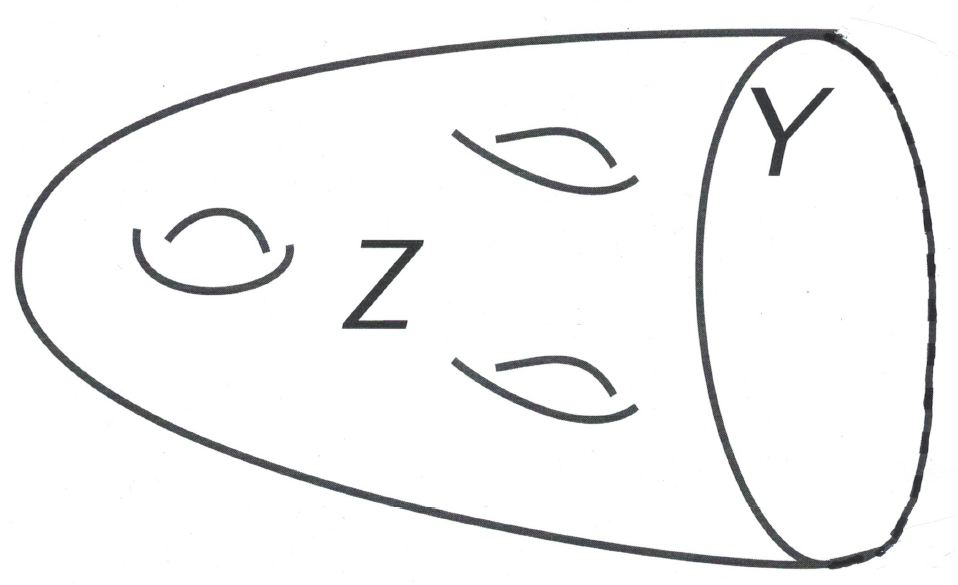}\qquad\qquad
\includegraphics[height=1in]{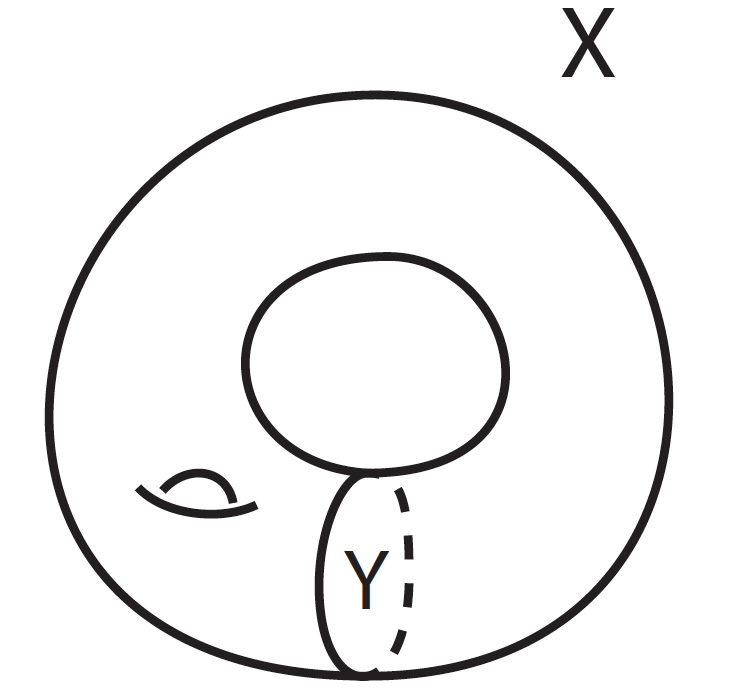}\qquad \qquad
\includegraphics[height=1in]{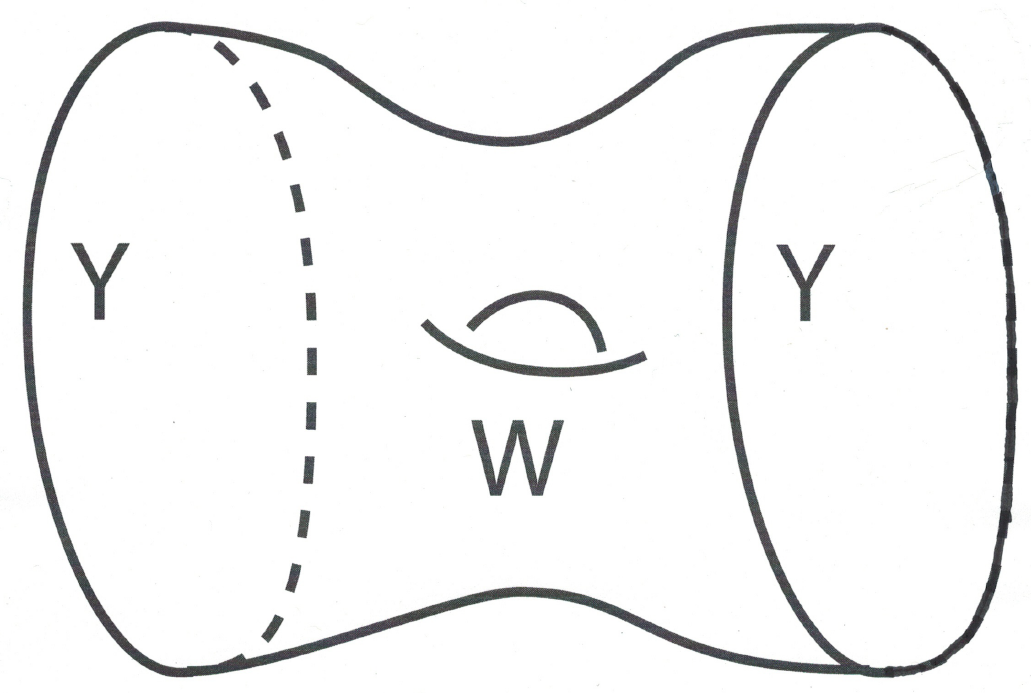}
\caption{Pieces of an end-periodic manifold}
\label{fig:epc}
\end{figure}

\noindent
If $W_k$ are isometric copies of $W$, then we can attach $X_1=\bigcup_{k\ge 0} W_k$
to the boundary component $Y$ of $Z$, forming the \emph{end-periodic manifold $Z_\infty$} (Figure \ref{fig:ep}). Often in the paper, we also deal with manifolds 
with more
than one periodic end.

\begin{figure}[h]
\includegraphics[height=2in]{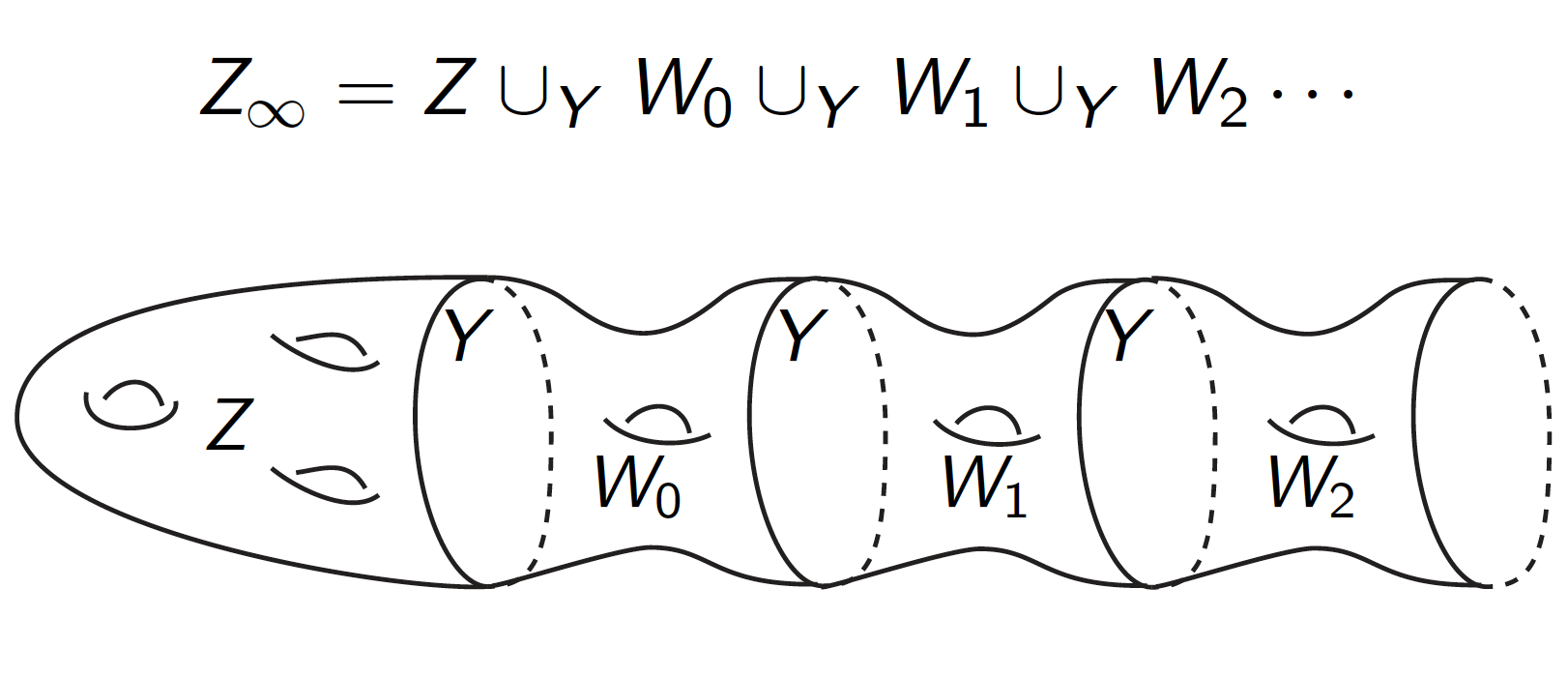}\\
\caption{End-periodic manifold}
\label{fig:ep}
\end{figure}

The motivations for considering such manifolds are from gauge theory;  it was Taubes \cite{Taubes} who originally developed the analysis of end-periodic elliptic operators on end-periodic manifolds, and successfully calculated the index of the end-periodic anti-self dual operator in Yang-Mills theory.

We adapt the results by Higson-Roe \cite{HigsonRoe}, using end-periodic \(K\)-homology,
 to obtain obstructions to the existence of positive scalar curvature 
metrics in terms of end-periodic eta invariants (see section \ref{end-per-rho}) that were defined by Mrowka-Ruberman-Saveliev \cite{MRS} for even dimensional manifolds, using the b-trace approach of Melrose \cite{Melrose}. These obstructions are for the compact manifold \(X\), and not the end-periodic manifold \(Z_\infty\); the end-periodic manifold is only a tool used to obtain the obstructions. This is established in section \ref{PSC}.
Roughly speaking, end-periodic \(K\)-homology is an analog of geometric \(K\)-homology, where the 
representatives have in addition, a choice of degree 1 cohomology class determining the 
codimension 1 submanifold. It is defined and studied in Section \ref{end-per-K}.

We also adapt the results by Botvinnik-Gilkey \cite{BG}, using end-periodic bordism, to obtain 
results on the number of components of the moduli space of Riemannian metrics of positive 
 scalar curvature 
metrics in terms of end-periodic eta invariants. Such results have been obtained by
Mrowka-Ruberman-Saveliev \cite{MRS}, and the introduction of end-periodic bordism provides a conceptualisation of their approach. Again, the information on path components is for the compact manifold \(X\), and the end-periodic manifold is but a means to obtaining this information.
End-periodic bordism is defined and studied in Section \ref{end-per-B}.

In Section \ref{structure} we define the end-periodic analogues of the structure groups of Higson and Roe, and study the end-periodic rho invariant on these groups. 

Section \ref{PSC} contains the applications to positive scalar curvature, using the established end-periodic \(K\)-theory and end-periodic spin bordism of the previous sections.

In Section \ref{rep-var} we give a  proof of the vanishing of the end-periodic rho invariant of the twisted Dirac operator with coefficients in a flat Hermitian vector bundle on a compact even dimensional Riemannian spin manifold $X$ of positive scalar curvature using the representation variety of 
$ \pi_1(X)$. 

It seems to be a general theme that for any geometrically defined homology theory, there is an analogous theory tailored to the setting of end-periodic manifolds, and that this end-periodic theory is isomorphic to the original geometric theory in a natural way. These isomorphisms are built on the foundation of Poincar{\'e} duality. \\

\noindent{\em {Acknowledgements.}} 
M.H. acknowledges an M.Phil. scholarship funding by the University of Adelaide.
V.M. acknowledges funding by the Australian Research Council, through Discovery Project
DP170101054. Both authors thank Jonathan Rosenberg for the explanation of the concordance is isotopy conjecture for PSC metrics and the relevance of the minimal hypersurfaces result of Schoen-Yau \cite{SY} to section 6.2, 
and Nikolai Saveliev for helpful feedback.

\section{End-Periodic \(K\)-homology}\label{end-per-K}
\subsection{Review of \(K\)-homology}

We begin by reviewing the definition of \(K\)-homology of Baum and Douglas \cite{BaumDouglas}, using the \((M,S,f)\)-formulation introduced by Keswani \cite{Keswani99}, and used by Higson and Roe \cite{HigsonRoe}. \begin{definition} A \emph{\(K\)-cycle} for a discrete group \(\pi\) is a triple \((M,S,f)\), where \(M\) is a compact oriented odd-dimensional Riemannian manifold, \(S\) is a smooth Hermitian bundle over \(M\) with Clifford multiplication \(c:TM\to\text{End}(S)\), and \(f:M\to B\pi\) is a continuous map to the classifying space of \(\pi\). \end{definition} Such a bundle \(S\) with the above data is called a \emph{Dirac bundle}. We remark that \(M\) may be disconnected, and that its connected components are permitted to have different odd dimensions.

\begin{definition} Two \(K\)-cycles \((M,S,f)\) and \((M',S',f')\) for \(B\pi\) are said to be \emph{isomorphic} if there is an orientation preserving diffeomorphism \(\varphi:M\to M'\) covered by an isometric bundle isomorphism \(\psi:S\to S'\) such that \[\psi\circ c_M(v)=c_{M'}(\varphi_*v)\circ \psi\] for all \(v\in TM\), and such that \(f'\circ\varphi=f\). \end{definition}

A \emph{Dirac operator} for the cycle \((M,S,f)\) is any first order linear partial differential operator \(D\) acting on smooth sections of \(S\) whose principal symbol is the Clifford multiplication. That is to say, for any smooth function \(\phi:M\to\mathbb{R}\) one has \[[D,\phi]=c(\text{grad}\,\phi):\Gamma(S)\to\Gamma(S).\]

The \(K\)-homology group \(K_1(B\pi)\) will consist of geometric \(K\)-cycles for \(\pi\) modulo an equivalence relation, which we will now describe.

\begin{definition}\label{boundary} A \(K\)-cycle \((M,S,f)\) is a \emph{boundary} if there exists a compact oriented even-dimensional manifold \(W\) with boundary \(\partial W=M\) such that:\vspace{2mm} \begin{enumerate}[(a)] \setlength\itemsep{2mm}
\item \(W\) is isometric to the Riemannian product \((0,1]\times M\) near the boundary.
\item There is a \(\mathbb{Z}_2\)-graded Dirac bundle over \(W\) that is isomorphic to \(S\oplus S\) in the collar with Clifford multiplication given by \[c_W(v)=\left(\begin{array}{cc}
0 & c_M(v) \\
c_M(v) & 0
\end{array} \right),\quad\quad c_W(\partial_t)=\left(\begin{array}{cc}
0 & -I \\
I & 0
\end{array}\right)\] for \(v\in TM\).
\item The map \(f:M\to B\pi\) extends to a continuous map \(f:W\to B\pi\).
\end{enumerate}\end{definition} 

\begin{remark}\label{orientation}
Our orientation convention for boundaries is the following: If \(W\) is an oriented manifold with boundary \(\partial W\) then the orientation on \(W\) at the boundary is given by the outward unit normal followed by the orientation of \(\partial W\). The isometry in part (a) is required to be orientation preserving.
\end{remark}

 We define the \emph{negative} of a \(K\)-cycle \((M,S,f)\) to be \((-M,-S,f)\), where \(-M\) is \(M\) with its orientation reversed, and \(-S\) is \(S\) with the negative Clifford multiplication \(c_{-S}=-c_S\). Two \(K\)-cycles \((M,S,f)\) and \((M',S',f')\) are \emph{bordant} if the disjoint union \((M,S,f)\amalg(-M',-S',f')\) is a boundary, and we write \((M,S,f)\sim(M',S',f')\). This is the first of the relations defining \(K\)-homology; there are two more to define:\vspace{2mm} \begin{enumerate}[(1)]
  \setlength\itemsep{2mm}

\item \emph{Direct sum/disjoint union}: \[(M,S_1,f)\amalg(M,S_2,f)\sim(M,S_1\oplus S_2,f).\]

\item \emph{Bundle modification}: Let \((M,S,f)\) be a \(K\)-cycle. If \(P\) is a principal \(SO(2k)\)-bundle over \(M\), we define \[\hat{M}=P\times_{\rho} S^{2k}.\] Here \(\rho\) denotes the action of \(SO(2k)\) on \(S^{2k}\) given by the standard embedding of \(SO(2k)\) into \(SO(2k+1)\). The metric on \(\hat{M}\) is any metric agreeing with that of \(M\) on horizontal tangent vectors and with that of \(S^{2k}\) on vertical tangent vectors. The map \(\hat{f}:\hat{M}\to B\pi\) is the composition of the projection \(\hat{M}\to M\) and \(f:M\to B\pi\). Over \(S^{2k}\) is an \(SO(2k)\)-equivariant vector bundle \(C\ell_\theta(S^{2k})\subset C\ell(TS^{2k})\), defined as the \(+1\) eigenspace of the \emph{right} action by the oriented volume element \(\theta\) on the Clifford bundle \(C\ell(TS^{2k})\). The \(SO(2k)\)-equivariance of this bundle implies that it lifts to a well defined bundle over \(\hat{M}\). We thus define the bundle \[\hat{S}=S\otimes C\ell_\theta(S^{2k})\] over \(\hat{M}\). Clifford multiplication on \(\hat{S}\) is given by \[ c(v)=\begin{cases}
c_M(v)\otimes \epsilon & \text{if } v \text{ is horizontal}, \\
I\otimes c_{S^{2k}}(v) & \text{if } v \text{ is vertical},
\end{cases} \] where \(\epsilon\) is the grading element of the Clifford bundle over \(S^{2k}\). The \(K\)-cycle \((\hat{M},\hat{S},\hat{f})\) is called an \emph{elementary bundle modification} of \((M,S,f)\), and we write \((M,S,f)\sim(\hat{M},\hat{S},\hat{f})\). We remark also that individual bundle modifications are allowed to be made on connected components of \(M\).
\end{enumerate}

\begin{remark}\label{rem:tensorDirac}
If \(D\) is a given Dirac operator for the cycle \((M,S,f)\), then there is a preferred choice of Dirac operator for an elementary bundle modification \((\hat{M},\hat{S},\hat{f})\) of \((M,S,f)\). If \(D_\theta\) denotes the \(SO(2k)\)-equivariant Dirac operator acting on \(C\ell(S^{2k})\), then the Dirac operator on \(S\otimes C\ell_\theta(S^{2k})\) is \[\hat{D}=D\otimes\epsilon+I\otimes D_\theta\] where \(\epsilon\) is the grading element of \(C\ell_\theta(S^{2k})\).
\end{remark}

\begin{definition}The \emph{\(K\)-homology group} \(K_1(B\pi)\) is the abelian group of \(K\)-cycles modulo the equivalence relation generated by isomorphism of cycles, bordism, direct sum/disjoint union, and bundle modification. The addition of equivalence classes of \(K\)-cycles is given by disjoint union \[(M,S,f)\amalg(M',S',f')=(M\amalg M',S\amalg S',f\amalg f').\] \end{definition} One must of course check that this operation descends to a well-defined binary operation on \(K\)-homology which satisfies the group axioms. The details are straightforward.

\begin{remark}
There is another group \(K_0(B\pi)\) defined in terms of even-dimensional cycles, which is well suited to the original Atiyah-Singer index theorem. We will not need it here.\\
\end{remark}

\subsection{Definition of End-Periodic \(K\)-homology}

With the above definition of \(K\)-homology reviewed, we now adapt the definition to the setting of manifolds with periodic ends. 

\begin{definition} An \emph{end-periodic \(K\)-cycle}, or simply a \emph{\(K^\ep\)-cycle} for a discrete group \(\pi\) is a quadruple \((X,S,\gamma,f)\), where \(X\) is a compact oriented even-dimensional Riemannian manifold, \(S=S^+\oplus S^-\) is a \(\mathbb{Z}_2\)-graded Dirac bundle over \(X\), \(\gamma\in H^1(X,\mathbb{Z})\) is a cohomology class whose restriction to each connected component of \(X\) is primitive, and \(f\) is a continuous map \(f:X\to B\pi\). \end{definition} 

The \(\mathbb{Z}_2\)-graded structure of \(S\) includes a Clifford multiplication by tangent vectors to \(X\) which swaps the positive and negative sub-bundles. Again, the manifold \(X\) is allowed to be disconnected, with the connected components possibly having different even dimensions. Note that the definition of a \(K^\ep\)-cycle imposes topological restrictions on \(X\), namely each connected component of \(X\) must have non-trivial first cohomology in order for the class \(\gamma\) to be primitive on each component.

\begin{definition}\label{ep-isom}
Two \(K^\ep\)-cycles \((X,S,\gamma,f)\) and \((X',S',\gamma',f')\) are \emph{isomorphic} if there exists an orientation preserving diffeomorphism \(\varphi:X\to X'\) which is covered by a \(\mathbb{Z}_2\)-graded isometric bundle isomorphism \(\psi:S\to S'\) such that \[\psi\circ c_X(v)=c_{X'}(\varphi_*v)\circ\psi\] for all \(v\in TX\). The diffeomorphism \(\varphi\) must additionally satisfy \(\varphi^*(\gamma')=\gamma\), and \(f'\circ\varphi=f\).
\end{definition}

We now define what it means for a \(K^\ep\)-cycle \((X,S,\gamma,f)\) to be a boundary. First, let \(Y\subset X\) be a connected codimension-1 submanifold that is Poincar{\'e} dual to \(\gamma\). The orientation of \(Y\) is such that for all closed forms \(\alpha\) of codimension 1 (over each component of \(X\)), \[\int_Y\iota^*(\alpha)=\int_X\gamma\wedge\alpha,\] where \(\iota:Y\to X\) is the inclusion and we abuse notation by writing \(\gamma\) for what is really a closed \(1\)-form representing the cohomology class \(\gamma\). In other words, the orientation of \(Y\) is such that the signs of the above two integrals always agree. Now, cut \(X\) open along \(Y\) to obtain a compact manifold \(W\) with boundary \(\partial W= Y\amalg-Y\), with our boundary orientation conventions as in Remark \ref{orientation}. Glue infinitely many isometric copies \(W_k\) of \(W\) end to end along \(Y\) to obtain the complete oriented Riemannian manifold \(X_1=\bigcup_{k\geq0}W_k\) with boundary \(\partial X_1=-Y\). Pull back the Dirac bundle \(S\) on \(X\) to get a \(\mathbb{Z}_2\)-graded Dirac bundle on \(X_1\), also denoted \(S\), and pull back the map \(f\) to get a map \(f:X_1\to B\pi\). 

\begin{definition}\label{ep-boundary} The \(K^\ep\)-cycle \((X,S,\gamma,g)\) is a \emph{boundary} if there exists a compact oriented Riemannian manifold \(Z\) with boundary \(\partial Z=Y\), which can be attached to \(X_1\) along \(Y\) to form a complete oriented Riemannian manifold \(Z_\infty=Z\cup_YX_1\), such that the bundle \(S\) extends to a \(\mathbb{Z}_2\)-graded Dirac bundle on \(Z_\infty\) and the map \(f\) extends to a continuous map \(f:Z_\infty\to B\pi\).  \end{definition}

\begin{remark}
Being a boundary is clearly independent of the choice of \(Y\); if \(Y'\) is another choice of submanifold Poincar{\'e} dual to \(\gamma\) we simply embed \(Y'\) somewhere in the periodic end of \(Z_\infty\), and take \(Z'\) to be the compact piece in \(Z_\infty\) bounded by \(Y'\).
\end{remark}

\begin{definition} The manifold \(Z_\infty\) from Definition \ref{ep-boundary} is called an \emph{end-periodic} manifold. It is convenient to say the end is \emph{modelled} on \((X,\gamma)\), or sometimes just \(X\) if \(\gamma\) is understood. Any object on \(Z_\infty\) whose restriction to the periodic end \(X_1\) is the pullback of an object from \(X\) is called \emph{end-periodic}. For example, the bundle \(S\), the map \(f\), and the metric on \(Z_\infty\) in the previous definition are all end-periodic. \end{definition}

\begin{remark}
We allow end-periodic manifolds to have multiple ends. This situation arises when the manifold \(X\), on which the end of \(Z_\infty\) is modelled, is disconnected.
\end{remark}

The \emph{negative} of a \(K^\ep\)-cycle \((X,S,\gamma,f)\) is simply \((X,S,-\gamma,f)\). This is so that the disjoint union of a \(K^\ep\)-cycle with its negative is a boundary---it is clear that the \(\mathbb{Z}\)-cover \(\tilde{X}\) of \(X\) corresponding to \(\gamma\) is an end-periodic manifold with end modelled on \((X\amalg X,\gamma\amalg-\gamma)\). The definitions of bordism and direct sum/disjoint union are exactly the same as before, with the class \(\gamma\) left unchanged. In the case of bundle modification, the class \(\hat{\gamma}\) on \(\hat{X}=P\times_\rho S^{2k}\) is the pullback of \(\gamma\) by the projection \(p:\hat{X}\to X\), and we endow the tensor product bundle \(S\otimes C\ell_\theta(S^{2k})\) with the standard tensor product grading of \(\Z_2\)-graded modules. There is also one more relation we define which relates the orientation on \(X\) to the one-form \(\gamma\): \[(X,S,-\gamma,f)\sim(-X,\Pi(S),\gamma,f)\] where \(-X\) is \(X\) with the reversed orientation and \(\Pi(S)\) is \(S\) with its \(\mathbb{Z}_2\)-grading reversed. We call this relation \emph{orientation/sign}, as it links the orientation on \(X\) to the sign of \(\gamma\). The need for this relation will become apparent in (2) of the proof of Lemma \ref{map1}.

\begin{definition} The \emph{end-periodic \(K\)-homology group}, \(K_1^\text{ep}(B\pi)\), is the abelian group consisting of \(K^\ep\)-cycles up to the equivalence relation generated by isomorphism of \(K^\ep\)-cycles, bordism, direct sum/disjoint union, bundle modification, and orientation/sign. Addition is given by disjoint union of cycles \[(X,S,\gamma,f)\amalg(X',S',\gamma',f')=(X\amalg X',S\amalg S',\gamma\amalg\gamma',f\amalg f').\]\end{definition}

\begin{remark}
As for \(K\)-homology we could also define the group \(K_0^\ep(B\pi)\) using odd-dimensional \(K^\ep\)-cycles, although we will not pursue this here.\\
\end{remark}

\subsection{The isomorphism}

We will now show that there is a natural isomorphism \(K_1(B\pi)\cong K_1^\ep(B\pi)\).

First we describe the map  \(K_1(B\pi)\to K_1^\ep(B\pi)\). Let \((M,S,f)\) be a \(K\)-cycle for \(B\pi\). Define \(X=S^1\times M\) an even dimensional manifold with the product orientation and Riemannian metric, the Dirac bundle \(S\oplus S\to X\) with Clifford multiplication as in (b) of Definition \ref{boundary}, \(\gamma=d\theta\in H^1(X,\mathbb{Z})\) the standard generator of the first cohomology of \(S^1\), and \(f:X\to B\pi\) the extension of \(f:M\to B\pi\). We map the equivalence class of \((M,S,f)\) in \(K_1(B\pi)\) to the equivalence class of \((S^1\times M,S\oplus S,d\theta,f)\) in \(K_1^\ep(\pi)\). \begin{lemma}\label{map1} The map sending a cycle \((M,S,f)\) to the end-periodic cycle \((S^1\times M,S\oplus S,d\theta,f)\) descends to a well-defined map of \(K\)-homologies.\end{lemma} 

\begin{proof} It must be checked that each of the relations defining \(K_0(B\pi)\) are preserved by this map.\vspace{2mm} \begin{enumerate}[(1)]
 \setlength\itemsep{3mm}

\item \textbf{Boundaries:} Let \((M,S,f)\) be a boundary. Then we have a compact manifold \(W\) with boundary \(\partial W=M\) satisfying conditions (a) and (b) in Definition \ref{boundary}. To show that \((S^1\times M,S\oplus S,d\theta,f)\) is a boundary, we attach \(W\) to the half-cover \(X_1=\mathbb{R}_{\geq0}\times M\) to obtain a Riemannian manifold \(Z_\infty\). Over \(X_1\) is the bundle \(S\oplus S\), and over \(W\) is a bundle isomorphic to \(S\oplus S\). We use the isomorphism to glue the bundles together and define \(S\oplus S\) over \(Z_\infty\). The assumptions on the Clifford multiplication imply that it extends over this bundle. Since the map \(f\) on M extends to \(W\), the map \(f\) on \(S^1\times M\) extends to \(Z_\infty\).

\item \textbf{Negatives:} The negative of \((M,S,f)\) is \((-M,-S,f)\), which maps to \((-S^1\times M,-S\oplus-S,d\theta,f)\). The negative of \((-S^1\times M,-S\oplus-S,d\theta,f)\) is \[(-S^1\times M,-S\oplus-S,-d\theta,f)\sim(S^1\times M,\Pi(-S\oplus-S),d\theta,f)\] by the orientation/sign relation. The only difference between this cycle and \((X,S\oplus S,d\theta,f)\) is that the Clifford multiplication is negative; Clifford multiplication by vectors tangent to \(M\) has become negative and reversing the \(\mathbb{Z}_2\)-grading has caused \(\partial_\theta\) to act negatively. This cycle is isomorphic to \[(S^1\times M,S\oplus S,d\theta,f)\] via the identity map \(\varphi:M\to M\) and the isometric bundle isomorphism \(\psi:-S\oplus-S\to S\oplus S\), \(\psi(s\oplus t)=c(\omega)(s\oplus t)\), where \(\omega\) is the oriented volume element of \(S^1\times M\). Hence negatives are preserved by the mapping.

\item \textbf{Disjoint union:} Obvious.

\item \textbf{Bordism:} Since negatives map to negatives, boundaries map to boundaries, and disjoint union is preserved, it follows that bordism is also preserved.

\item \textbf{Direct sum/disjoint union:} Also obvious. 

\item \textbf{Bundle modification:} Let \((\hat{M},\hat{S},\hat{f})\) be an elementary bundle modification for \((M,S,f)\) associated to the principal \(SO(2k)\)-bundle \(P\to M\). We pullback \(P\) to a bundle over \(X=S^1\times M\), and use it to construct our bundle modification \((\hat{X},(S\oplus S)\string^,d\theta,f)\) of \((S^1\times M, S\oplus S,d\theta,f)\). It is clear that \(\hat{X}=S^1\times \hat{M}\). Now \(\hat{S}=S\otimes C\ell_\theta(S^{2k})\), so \[\hat{S}\oplus\hat{S}\cong(S\oplus S)\otimes C\ell_\theta(S^{2k})=(S\oplus S)\string^.\] It is straightforward yet tedious to verify that Clifford multiplication is preserved by this isomorphism. So the \(K^\ep\)-cycle obtained via bundle modification then mapping, is isomorphic to the \(K^\ep\)-cycle obtained by mapping then bundle modification.\qedhere \end{enumerate}\end{proof}

Now for the inverse map. Let \((X,S,\gamma,f)\) be an end-periodic cycle. Choose a submanifold \(Y\subset X\) Poincar{\'e} dual to \(\gamma\), oriented as in the paragraph after Definition \ref{ep-isom}. We map the cycle \((X,S,\gamma,f)\) to \((Y,S^+,f)\), where \(S^+\) and \(f\) are restricted to \(Y\). If \(\omega\) is an oriented volume form for \(Y\) then we let \(\partial_t\) be the unit normal to \(Y\) such that \(\partial_t\wedge\omega\) is the orientation on \(X\). The Clifford multiplication on \(S^+\) is then defined to be \[c_Y(v)=c_X(\partial_t)c_X(v)\] for \(v\in TY\). Note that this agrees with the conventions of (b) in Definition \ref{boundary}. One easily verifies that this indeed defines a Clifford multiplication on \(S^+\). 

\begin{lemma} The map sending an end-periodic cycle \((X,S,\gamma,f)\) to the cycle \((Y,S^+,f)\) described above, descends to a well-defined map of \(K\)-homologies.\end{lemma} 

\begin{proof} We must not only check that the relations defining end-periodic \(K\)-homology are preserved, but that the class in \(K\)-homology obtained is independent of the choice of \(Y\). \vspace{2mm}

\begin{enumerate}[(1)]
 \setlength\itemsep{3mm}

\item \textbf{Boundaries:} Let \((X,S,\gamma,f)\) be a boundary. Then there is a compact oriented manifold \(Z\) with boundary \(\partial Z=Y\) over which the \(\mathbb{Z}_2\)-graded Dirac bundle \(S\) and map \(f\) extend. We modify the metric near the boundary of \(Z\) to make it a product. It follows that the cycle \((Y,S^+,f)\) is a boundary.

\item \textbf{Choice of \(Y\):} Suppose \(Y_1\) and \(Y_2\) are submanifolds of \(X\) that are Poincar{\'e} dual to \(\gamma\). The class \(\gamma\) determines a \(\mathbb{Z}\)-cover \(\tilde{X}\) of \(X\), and \(Y_1,Y_2\) may be considered as submanifolds of this cover. Since both \(Y_1\) and \(Y_2\) are compact, they can be embedded in \(\tilde{X}\) so that they are disjoint. We delete the open subset of \(\tilde{X}\) lying outside of \(Y_1\) and \(Y_2\), and leave only the points in \(\tilde{X}\) between and including \(Y_1\) and \(Y_2\).\begin{figure}[h]
\includegraphics[height=1in]{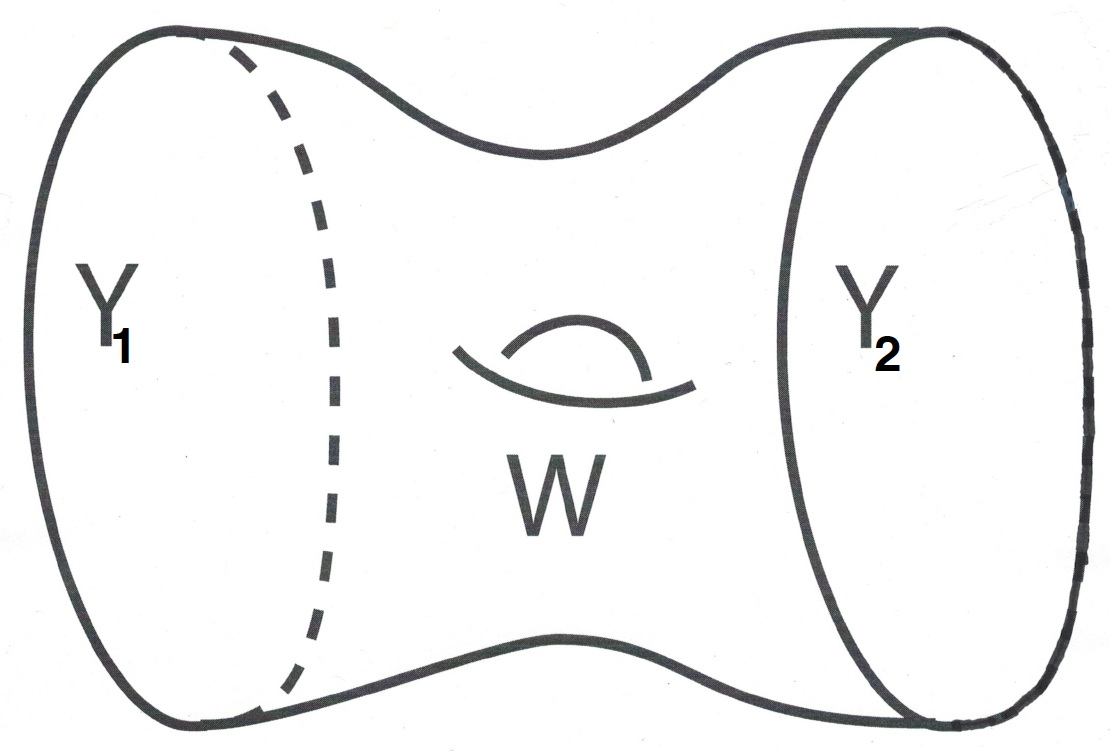}
\caption{Piece of an end-periodic manifold}
\label{fig:epc2}
\end{figure} We call the remaining manifold \(W\); it is a compact manifold with boundary \(\partial W=Y_1\amalg -Y_2\). We pull back the bundle \(S\) and the map \(f\) to \(W\), and modify the metric near the boundary so that it is a product. The result is that \(Y_1\amalg -Y_2\) is a boundary.

\item \textbf{Negatives:} Reversing the sign of \(\gamma\) changes the orientation of \(Y\). Clifford multiplication on \(Y\) also becomes negative, since changing the orientation on \(Y\) reverses the unit normal to \(Y\). Hence negatives of cycles map to negatives.

\item \textbf{Disjoint union:} Obvious.

\item \textbf{Bordism:} Since boundaries map to boundaries, negatives map to negatives, and disjoint union is preserved, it follows that bordism is also preserved.

\item \textbf{Direct sum/disjoint union:} Obvious.

\item \textbf{Orientation/sign:} From (3) in this proof, the \(K\)-cycle obtained from \((X,S,-\gamma,f)\) is the negative of the cycle \((Y,S^+,f)\). Now consider the \(K\)-cycle obtained from \((-X,\Pi(S),\gamma,f)\). Reversing the orientation on \(X\) will also reverse it on \(Y\). Instead of \(S^+\), we now take \(S^-\) with Clifford multiplication \[c_{S^-}(v)=c(-\partial_t)c(v)=-c(\partial_t)c(v)\] where \(v\in TY\) and \(-\partial_t\) is the unit normal to \(-Y\). We now show \((-Y,S^+,f)\) and \((-Y,S^-,f)\) are isomorphic. Let \(\omega\) be the oriented volume element of \(+Y\) (or \(-Y\), it does not matter) and define a map \(\psi:S^+\to S^-\) by \(\psi(s)=c(\omega)s\). Then \[\psi\circ c_{S^+}(v)=c_{S^-}(v)\circ\psi\] and the cycles are therefore isomorphic.

\item \textbf{Bundle modification:} Let \((\hat{X},\hat{S},\hat{\gamma},\hat{f})\) be an elementary bundle modification for \((X,S,\gamma,f)\), associated to the principal \(SO(2k)\)-bundle \(P\to X\). We restrict this principal bundle to \(Y\) and consider the corresponding bundle modification \((\hat{Y},\hat{S^+},\hat{f})\) for \((Y,S^+,f)\). It is clear that \(\hat{Y}\subset \hat{X}\) is Poincar{\'e} dual to \(\hat{\gamma}\). The bundle \[\hat{S}=S\otimes C\ell_\theta(S^{2k})\] has even part \[\hat{S}^+=(S^+\otimes C\ell_\theta^+(S^{2k}))\oplus (S^-\otimes C\ell_\theta^-(S^{2k})),\] while over \(\hat{Y}\) we have the bundle \[\hat{S^+}=S^+\otimes C\ell_\theta(S^{2k}).\] Identifying \(S^+\) with \(S^-\) via the isomorphism \(c(\partial_t)\), we see that \(\hat{S}^+\cong \hat{S^+}\). It is routine to check that the Clifford multiplications are preserved under this isomorphism.\qedhere
\end{enumerate} \end{proof}

\begin{theorem}\label{isom}
The above maps between \(K\)-homologies define an isomorphism of groups \(K_1(B\pi)\cong K_1^\ep(B\pi)\).
\end{theorem}

\begin{proof} We must check that the above maps on \(K\)-homologies are inverse to each other. If we begin with a cycle \((M,S,f)\), this maps to \((S^1\times M,S\oplus S,d\theta,f)\). Mapping this again, we get \((M,S,f)\) back, so this direction is easy. Now suppose we begin with a cycle \((X,S,\gamma,f)\). This maps to \((Y,S^+,f)\) which then maps to \((S^1\times Y,S^+\oplus S^+,d\theta,f)\). We will show this cycle is bordant to the original cycle \((X,S,\gamma,f)\). Consider the half cover \(X_1\) of \(X\) obtained using \(-\gamma\). Near the boundary, this is diffeomorphic to a product \((-\delta,0]\times Y\). The half cover of \(S^1\times Y\) obtained from \(d\theta\) is \(\mathbb{R}_{\geq0}\times Y\). The two half covers clearly glue together to produce and end-periodic manifold with two ends. 
\begin{figure}[h]
\includegraphics[height=1.2in]{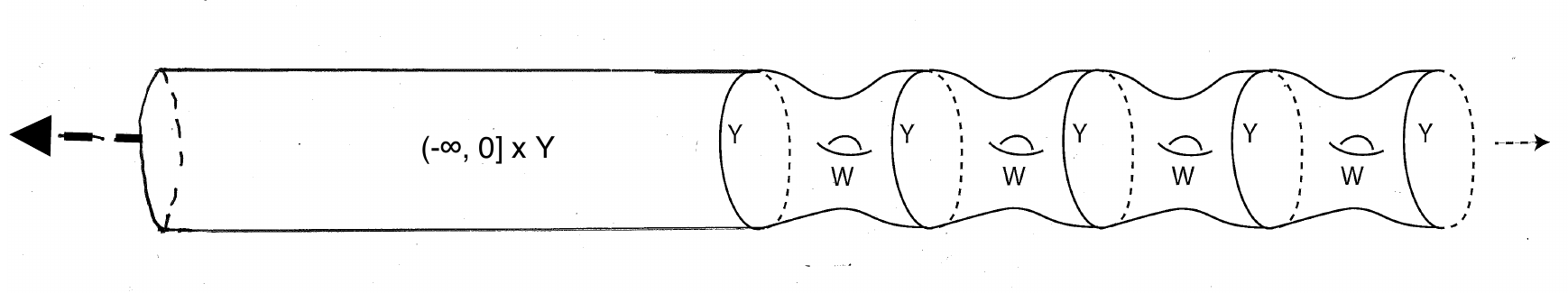}
\caption{End-periodic manifold with two ends}
\label{fig:epc2}
\end{figure}
The Dirac bundles and maps to \(B\pi\) extend over this manifold, and hence the two cycles are bordant. \end{proof}

\section{Relative eta/rho invariants}\label{end-per-rho}

In this section, we use the end-periodic eta invariant of MRS to define homomorphisms from the end-periodic \(K\)-homology group \(K_1^\ep(B\pi)\) to \(\mathbb{R}/\mathbb{Z}\). Any pair of unitary representations \(\sigma_1,\sigma_2:\pi\to U(N)\) will determine such a homomorphism, and we see that this homomorphism agrees with that constructed in Higson-Roe \cite{HigsonRoe} under the natural isomorphism \(K_1(B\pi)\cong K_1^\ep(B\pi)\).

\subsection{Rho invariant for \(K\)-homology}

Let \((M,S,f)\) be a \(K\)-cycle. Any Dirac operator for this cycle is a self-adjoint elliptic first order operator on \(S\), and so has a discrete spectrum of real eigenvalues. The \emph{eta function} of this operator is defined to be the sum over the non-zero eigenvalues of \(D\) \[\eta(s)=\sum_{\lambda\neq0}\text{sign}(\lambda)|\lambda|^{-s},\] which converges absolutely for \(\text{Re}(s)\) sufficiently large. It is a theorem of Atiyah, Patodi and Singer (APS) that this function admits a meromorphic continuation to the complex plane, and that this continuation takes a finite value \(\eta(0)\) at the origin. The \emph{eta invariant} of the chosen Dirac operator \(D\) is by definition \begin{equation}\label{eq:eta}
\eta(D)=\frac{\eta(0)-h}{2}
\end{equation} where \(h=\dim\ker(D)\) is the multiplicity of the zero eigenvalue.

The eta invariant plays a central role in the Atiyah-Patodi-Singer index theorem, appearing as a correction term for the boundary. Suppose \(W\) is an even dimensional manifold with boundary \(\partial W=M\), equipped with a Dirac bundle satisfying the conditions of Definition \ref{boundary}. Further, suppose we have a Dirac operator \(D(W)\) on \(W\) so that \begin{equation}\label{eq:W-Dirac}
D(W)=\begin{pmatrix}
0 & -\partial_t+D\\
\partial_t+D & 0
\end{pmatrix}
\end{equation} in a product neighbourhood of the boundary, where \(D\) is the Dirac operator on \(M\). In this instance we say that \(D(W)\) \emph{bounds} \(D\). Then the APS index theorem \cite{APS1} states \begin{equation}\label{eq:APS}\Ind_\text{APS}D^+(W)=\int_W{\bf I}(D^+(W))-\eta(D).\end{equation} The left-hand side is the index of \(D^+(W)\) with respect to a certain global boundary condition -- the projection onto the non-negative eigenspace of \(D\) must vanish. The integrand \({\bf I}(D^+(W))\) is the constant term in the asymptotic expansion of the supertrace of the heat operator for \(D^+(W)\), called the \emph{index form} of the Dirac operator.

\begin{remark}\label{h-sign}
In equation \eqref{eq:APS}, the eta invariant is as in \eqref{eq:eta}, where the sign of the term \(h=\dim\ker D\) is negative. This is contingent on the orientation of \(M\) being consistent with the boundary orientation inherited from \(W\). If the orientations are not compatible, then the sign of \(h\) is reversed in equation \eqref{eq:APS}.
\end{remark}

The map \(f\) in the cycle \((M,S,f)\) determines a principle \(\pi\)-bundle over \(M\). Given a representation \(\sigma_1:\pi\to U(N)\), we can then form a flat vector bundle \(E_1\to M\) and twist the Dirac operator \(D\) on \(S\) to obtain a Dirac operator \(D_1\) acting on sections of \(S\otimes E_1\). Given a second representation \(\sigma_2:\pi\to U(N)\) we form another operator \(D_2\) on \(S\otimes E_2\) in the same way. \begin{definition}\label{def:rho}
The \emph{relative eta invariant}, or \emph{rho invariant} associated to the two unitary representations \(\sigma_1,\sigma_2:\pi\to U(N)\), the \(K\)-cycle \((M,S,f)\) for \(B\pi\), and the choice of Dirac operator \(D\) for the \(K\)-cycle, is defined to be \[\rho\,(\sigma_1,\sigma_2\,;M,S,f)=\eta(D_1)-\eta(D_2).\]
\end{definition} The eta invariant of an operator depends sensitively on the operator itself, whereas the relative eta invariant is much more robust. The following is a restatement of Theorem 6.1 from Higson-Roe \cite{HigsonRoe}, and is the reason for our omission of \(D\) in the above notation for the rho invariant.

\begin{theorem}\label{eta1}
The \emph{mod} \(\mathbb{Z}\) reduction of the rho invariant \(\rho(\sigma_1,\sigma_2\,;M,S,f)\) for representations \(\sigma_1,\sigma_2:\pi\to U(N)\), depends only on the equivalence class of \((M,S,f)\) in \(K_1(B\pi)\), and on \(\sigma_1,\sigma_2\). There is therefore a well-defined group homomorphism \[\rho\,(\sigma_1,\sigma_2):K_1(B\pi)\to\mathbb{R}/\mathbb{Z}.\]
\end{theorem}

The most complicated part of the proof is showing invariance under bundle modification. We will not repeat the full proof, however we will show invariance under bordism since the argument serves to motivate the end-periodic case. 

\begin{proof} Let \((M,S,f)\) be a boundary---we will show that the rho invariant \(\rho\,(\sigma_1,\sigma_2\,;M,S,f)\) vanishes modulo \(\mathbb{Z}\). Let \(W\) be as in Definition \ref{boundary} and let \(D(W)\) be a Dirac operator on \(W\) which bounds the Dirac operator \(D\) on \(M\). Since the map \(f\) to \(B\pi\) extends to \(W\), we find twisted Dirac operators \(D_1(W)\) and \(D_2(W)\) on \(W\) bounding the twisted operators \(D_1\) and \(D_2\) on \(M\). Applying the APS index theorem separately to these operators gives \begin{equation}\label{eq:app}
\Ind_\text{APS}D_i^+(W)=\int_W{\bf I}(D_i^+(W)) -\eta(D_i)
\end{equation} for \(i=1,2\). Since \(D_1(W)\) and \(D_2(W)\) are both twists of the same Dirac operator \(D(W)\) by flat bundles of dimension \(N\), we have \[{\bf I}(D_1^+(W))={\bf I}(D_2^+(W))=N\cdot{\bf I}(D^+(W)).\] Subtracting the two equations \eqref{eq:app} from each other therefore yields \[\rho\,(\sigma_1,\sigma_2\,;M,S,f)=\eta(D_1)-\eta(D_2)=\Ind_\text{APS}D_2^+(W)-\Ind_\text{APS}D_1^+(W)\] which is an integer. 

Now, consider the negative cycle \((-M,-S,f)\) for \((M,S,f)\). If \(D\) is a Dirac operator for \((M,S,f)\), then \(-D\) is a Dirac operator for \((-M,-S,f)\). From the definition of the eta invariant \eqref{eq:eta} and from Remark \ref{h-sign}, we see that \(\eta(-D)=-\eta(D)\). Finally, the eta invariant is clearly additive under disjoint unions of cycles. It follows that if two cycles are bordant, then their eta invariants agree modulo integers.  \end{proof}

Higson and Roe \cite{HigsonRoe} used this map on \(K\)-homology to obtain obstructions to positive scalar curvature for odd-dimensional manifolds. Our isomorphism of \(K\)-homologies will allow us to transfer their results to the even dimensional case.

\subsection{Index theorem for end-periodic manifolds \cite{MRS}}

In \cite{MRS}, Mrowka, Ruberman and Saveliev prove an index theorem for end-periodic Dirac operators on end-periodic manifolds, which generalises the Atiyah-Patodi-Singer index theorem. Rather than the eta invariant appearing as a correction term for the end, a new invariant called the \emph{end-periodic eta invariant} appears, and this new invariant agrees with the eta invariant of Atiyah-Patodi-Singer in the case of a cylindrical end. In this section, we review the end-periodic index theorem of MRS, and give the necessary definitions and theorems required to define the end-periodic rho invariants. There is nothing new here, so the reader who is already familiar with the MRS index theorem may safely skip to Section \ref{RmodZ}

Let \((X,S,\gamma,f)\) be a \(K^\ep\)-cycle, and let \(D(X)\) be a Dirac operator for the cycle. Let \(\tilde{X}\) be the \(\mathbb{Z}\)-cover associated to \(\gamma\), and let \(F:\tilde{X}\to\mathbb{R}\) be the map which covers the classifying map \(X\to S^1\) for the \(\mathbb{Z}\)-cover \(\tilde{X}\). Then \(F\) satisfies \(F(x+1)=F(x)+1\), where \(x+1\) denotes the image of \(x\in\tilde{X}\) under the fundamental covering translation. It follows that \(dF\) descends to a well-defined one-form on \(X\), also denoted \(dF\). Fixing a branch of the complex logarithm, define a family of operators \[D_z(X)=D(X)-\ln(z)\,c(dF)\] on \(X\), where \(c(dF)\) is Clifford multiplication by \(dF\), and \(z\in\mathbb{C}^*\). These are in fact the operators obtained by conjugating the Dirac operator on \(\tilde{X}\) with the \emph{Fourier-Laplace transform}---see Section 2.2 of \cite{MRS} for more details. The \emph{spectral set} of this family of operators is defined to be the set of \(z\) for which \(D_z(X)\) is not invertible. The spectral sets of the families \(D_z^\pm(X)\) are defined similarly.

Henceforth, we will take \(Z_\infty\) to be an end-periodic manifold with end modelled on \((X,\gamma)\). All objects on \(Z_\infty\) will be taken to be end-periodic, unless stated otherwise. Now, the Fredholm properties of the end-periodic operator \(D^+(Z_\infty)\) are linked to the spectral set of the family \(D_z^+(X)\). In fact, it follows from Lemma 4.3 of Taubes \cite{Taubes}, that \(D^+(Z_\infty)\) is Fredholm if and only if the spectral set of the family \(D_z^+(X)\) is disjoint from the unit circle \(S^1\subset\mathbb{C}\). Thus, a necessary (but not sufficient) condition for \(D^+(Z_\infty)\) to be Fredholm is that \(\Ind D^+(X)=0\).

\begin{definition}[\cite{MRS}]\label{def:ep-eta}
Suppose that the spectral set of the family \(D_z^+(X)\) is disjoint from the unit circle \(S^1\subset\mathbb{C}\). The \emph{end-periodic eta invariant} for the Dirac operator \(D^+(X)\) is then defined as \[\eta^\ep(D^+(X))=\frac{1}{\pi i}\int_0^\infty\oint_{|z|=1}\Tr(c(dF)\cdot D_z^+\exp(-tD_z^-D_z^+))\,\frac{dz}{z}\,dt,\] where the Dirac operators in the integral are on \(X\), and the contour integral over the unit circle is taken in the anti-clockwise direction.
\end{definition}

\begin{remark}
There is an equivalent definition of the eta invariant in terms of the von Neumann trace---see Proposition 6.2 of \cite{MRS}, also \cite{Atiyah76} for information on the von Neumann trace.
\end{remark}

Suppose $X=S^1 \times Y$, where $Y$ is a compact oriented odd dimensional manifold, and $X$ is endowed with the product Riemannian metric. Assume the Dirac operator \(D(X)\) on \(X\) takes the form of that in the RHS of equation \eqref{eq:W-Dirac}, with \(D\) being the Dirac operator on \(Y\). Then it is shown in section 6.3 \cite{MRS} that for \(dF=d\theta\),
\[\eta^\ep(D^+(X)) = \eta(D).\]

We now state the end-periodic index theorem of Mrowka, Ruberman and Saveliev, in the case when the end-periodic operator \(D^+(Z_\infty)\) is Fredholm. Recall that for \(D^+(Z_\infty)\) to be Fredholm, it is necessary that \(\Ind D^+(X)=0\). The Atiyah-Singer index theorem then implies that the index form \({\bf I}(D^+(X))\) is exact, so one can find a form \(\omega\) on \(X\) satisfying \(d\omega={\bf I}(D^+(X))\).

\begin{theorem}[MRS Index Theorem, Theorem A, \cite{MRS}] \label{thmA,MRS}
Suppose that the end-periodic operator $D^+(Z_\infty)$ is Fredholm, and choose a form $\omega$ on $X$ such that 
$d\omega = {\bf I}(D^+(X))$. Then 
\smallskip
\begin{equation}\label{E:indexintro}
\Ind D^+ (Z_\infty) = \int_Z {\bf I}(D^+(Z))\; -
\int_Y\omega + \int_X dF\wedge\omega\; - \frac 1 2\,\eta^\ep (X).
\end{equation}
\end{theorem}

\begin{remarks}
The form $\omega$ is called the \emph{transgression class} -- see Gilkey \cite{Gilkey}, page 306 for more details. In the case that the metric is a product near \(Y\), one can choose \(F\) so that the two integrals involving the transgression class cancel, leaving a formula similar to the original APS formula.
The theorem reduces to the APS index theorem \cite{APS1} when $Z_\infty$ only has cylindrical ends.
\end{remarks}

When \(D^+(Z_\infty)\) is \emph{not} Fredholm, Mrowka, Ruberman and Saveliev are still able to prove an index theorem under the assumptions that the spectrum of the family \(D_z^+(X)\) is discrete, which in particular implies \(\text{Ind}D^+(X)=0\). This is analogous to the case in the APS index theorem when the Dirac operator \(D\) on the boundary has a non-zero kernel, and the correction \(h=\dim\ker D\) appears in the formula.

The key is to introduce the \emph{weighted Sobolev spaces} on \(Z_\infty\) as follows. First recall that the Sobolev space \(L^2_k(Z_\infty,S)\) for an integer \(k\geq0\), is defined as the completion of \(C_0^\infty(Z_\infty,S)\) in the norm \[\|u\|_{L^2_k(Z_\infty,S)}^2=\sum_{j\leq k}\int_{Z_\infty}|\nabla^{j}u|^2\] for a fixed choice of end-periodic metric and compatible end-periodic Clifford connection on \(Z_\infty\). Now, restrict the upstairs covering map \(F:\tilde{X}\to\mathbb{R}\) to the half-cover \(X_1=\bigcup_{k\geq0}W_k\), and choose an extension of this map to \(Z_\infty\), which we continue to denote \(F\). Given a weight $\delta \in \R$ and an integer $k\geq0$, we say that $u \in L^2_{k,\delta}\, (Z_\infty,S)$ if $e^{\delta F} u \in L^2_k\,(Z_\infty,S)$. Define the \(L^2_{k,\delta}\)-norm by
\[ 
\| u \|_{L^2_{k,\delta}\,(Z_\infty,S)} = \|\,e^{\delta F}\,u \|_{L^2_k\,(Z_\infty,S)}.  
\] 
It is easy to check that up to equivalence of norms, this is independent of the choice of extension of \(F\) to \(Z_\infty\), since the region over which we are choosing an extension is compact. The spaces \(L^2_{k,\delta}(Z_\infty,S)\) are all complete in this norm, and the operator $D^+(Z_\infty)$ extends to a bounded operator 
\begin{equation}\label{eq:dirac}
D^+(Z_\infty): L^2_{k+1,\delta}\, (Z_\infty,S^+) \to L^2_{k,\delta}\, (Z_\infty,S^-)
\end{equation} for every \(k\) and \(\delta\). The following theorem of Taubes \cite{Taubes} classifies Fredholmness of the operator \eqref{eq:dirac} in terms of the family \(D_z^+(X)=D^+(X)-\ln(z)\,c(dF)\).

\begin{lemma}[Lemma 4.3 \cite{Taubes}]\label{Taubes-fred1}
The operator $D^+(Z_\infty):L^2_{k+1,\delta}\, (Z_\infty,S^+) \to L^2_{k,\delta}\, (Z_\infty,S^-) $ is Fredholm if and only if the operators 
$D^+_z(X)$ are invertible for all $z$ on the circle $|z| = e^\delta$.
\end{lemma}

The usual \(L^2\)-case corresponds to the weighting \(\delta=0\), and hence we see by setting \(z=1\):

\begin{corollary}\label{C:fred}
A necessary condition for the operator $D^+(Z_\infty)$ to be Fredholm is that $\Ind D^+(X) = 0$. 
\end{corollary}

The following result on the spectral set of the family is also due to Taubes, % Theorem 3.1 \cite{Taubes}, 
which suffices for our purposes.

\begin{theorem}[Theorem 3.1, \cite{Taubes}]\label{Taubes-fred2}
Suppose that $\,\Ind D^+ (X) = 0$ and that the map $c(dF): \ker D^+ (X) \to \ker D^- (X)$ is injective. Then the spectral set of the family $D^+_z (X)$ is a discrete subset of $\C^*$,
and the operator $D^+(Z_\infty)$ is a Fredholm operator. 
\end{theorem}

It follows that the operator $D^+(Z_\infty) $ acting on the Sobolev spaces of weight \(\delta\) is Fredholm for all but a closed discrete set of $\delta \in \R$.

\begin{remark}\label{discrete_asssumptions}
	There are two important instances where the hypothesis of Theorem \ref{Taubes-fred2} is satisfied: \begin{enumerate}
		\item When \(X=S^1\times M\) with the product metric, and the Dirac operator on \(X\) taking the form of equation \eqref{eq:W-Dirac}. In this case \(dF=d\theta\), and \(c(d\theta)\) is as in part (b) of Definition \ref{boundary}. This example shows that every class in \(K^\ep(B\pi)\) has a representative with discrete spectral set.
		\item When \(X\) is spin with positive scalar curvature and \(D^+(X)\) is the spin Dirac operator on \(X\) (or more generally, \(D^+(X)\) twisted by a flat bundle). In this case Lichnerowicz' vanishing theorem implies that \(\ker D^+(X)\) and \(\ker D^-(X)\) are trivial. In the applications to positive scalar curvature, we will always assume \(X\) to be spin, so that this assumption is satisfied.
	\end{enumerate}
%The notation $c(dF)$ stands for $c(\alpha)$, where $\alpha$ is a closed 1-form such that $[\alpha] = \gamma$. 
\end{remark}

Theorem C, \cite{MRS}
extends Theorem \ref{thmA,MRS}  to the non-Fredholm case that applies to operators such as the signature operator and is 
analogous to the extended $L^2$ case considered in \cite{APS1}. 

%Assume henceforth that the spectral set of the family $D^+_z (X)$ is discrete, which will be case if we choose $F$ to be an end-periodic Morse function. 

We allow for the case where the family has poles lying on the unit circle, in which case the operator \(D^+(X)\) is not Fredholm. By discreteness of the spectral set, the family \(D_z^+(X)\) has no poles for \(z\) sufficiently close to (but not lying on) the unit circle, and hence there is \(\epsilon>0\) such that for all \(0<\delta<\epsilon\) the operators \(D_z^+(Z_\infty)\) acting on the \(\delta\)-weighted Sobolev spaces are all Fredholm (see Lemma \ref{Taubes-fred1}). The index does not change under small variations of \(\delta\) in this region, and we denote it by \(\Ind_\text{MRS} D^+(Z_\infty)\). This is the regularised form of the index which appears in the full MRS index theorem.

There are two more quantities to define which appear in the full MRS index theorem. First of all, the end-periodic eta invariant in Definition \ref{def:ep-eta} is no longer well defined if the family \(D_z^+(X)\) has poles on the unit circle. Letting \(\epsilon>0\) be sufficiently small so that there are no poles in \(e^{-\epsilon}<|z|<e^{\epsilon}\) except for those with \(|z|=1\), define 
\begin{equation}\label{E:eta-ep}
\eta^{\ep}_\epsilon(D^+(X)) = \frac 1 {\pi i}\,\int_0^{\infty}\oint_{|z| = e^{\epsilon}}\;
\Tr\left(df\cdot D^+_z \exp (-t (D^+_z)^* D^+_z)\right)\,\frac {dz} z\,dt,
\end{equation}

where the integral is taken to be the constant term of its asymptotic expansion in powers of \(t\). Define \[
\eta^\ep_{\pm} (D^+(X))\;=\; \lim_{\epsilon \to 0\pm}\,\eta^{\ep}_\epsilon(D^+(X)),
\] and \begin{equation}\label{E:etapm}
\eta^\ep(D^+(X))\;=\;\frac 1 2\;[\eta_+^\ep(D^+(X)) + \eta_-^\ep(D^+(X))].
\end{equation}

\smallskip\noindent
It is this incarnation of the eta invariant which will appear in the MRS index theorem. Since $(D^+_z)^* = D_z^-$ for $|z| = 1$ this definition of $\eta^\ep(X)$ agrees with Definition \ref{def:ep-eta} when there are no poles on the unit circle. 

%One can interpret $\eta^\ep(X)$ defined by \eqref{E:etapm} as a regularization of the series 
%\[
%\sum_{|z_k| \ne 1}\;\sign\, (\ln |z_k|)\cdot d(z_k),
%\]
%where the \(z_k\) are the poles of the family \(D_z^+(X)\) and \(d(z_k)\) is defined in the following paragraph; see Section 6.1 of MRS \cite{MRS} for more details. The equality $\eta^\ep(D_X^+) = \eta(D_Y)$  for product manifolds $X = S^1 \times Y$ continues to hold in the non-Fredholm case.

The last term to define is the analogue of \(h=\dim\ker D\) appearing in the APS index theorem. The family \(D_z^+(X)^{-1}\) is meromorphic, so if \(z\in S^1\) is a pole then it has some finite order \(m\). Define \(d(z)\), as in Section 6.3 of \cite{MRS2}, to be the dimension of the vector space solutions \((\varphi_1,\ldots,\varphi_m)\) to the system of equations \[\begin{cases}
D_z^+(X)\varphi_1=c(dF)\varphi_2 \\
\vdots\\
D_z^+(X)\varphi_{m-1}=c(dF)\varphi_m \\
D_z^+(X)\varphi_m=0.
\end{cases}\] For \(z\) not in the spectral set of the family \(D_z^+(X)\), we have \(d(z)=0\). The term \(h\) in the MRS index theorem is defined as the finite sum of integers \[h=\sum_{|z|=1}d(z).\]

\begin{remark}
The integers \(d(z)\) give a formula for the change in index when one varies the weight \(\delta\); if \(\Ind_\delta D^+(Z_\infty)\) denotes the index of \(D^+(Z_\infty)\) acting on the \(\delta\)-weighted Sobolev spaces, then one has for \(\delta<\delta'\) that \[
\Ind_{\,\delta} D^+(Z_\infty) - \Ind_{\,\delta'} D^+(Z_\infty)\; =\; 
\sum_{e^{\delta}< |z| <\, e^{\delta'}}\; d(z).
\]
\end{remark}

\begin{theorem}[MRS Index Theorem, Theorem C, \cite{MRS}] \label{thmC,MRS}
Suppose the spectral set of $D^+_z (X)$ is a discrete subset of $\C^*$, and let $\omega$ be a form on $X$ such that $d\omega = {\bf I}(D^+(X))$. Then 
\[
\Ind_\emph{MRS}D^+(Z_\infty)\;=\;\int_Z\; {\bf I}(D^+(Z)) - \int_Y\;\omega\, + 
\int_X\;dF\wedge \omega\, -\, \frac {h + \eta^\ep\,(D^+(X))}{2}.
\]
\end{theorem}

\subsection{End-periodic $\RE/\ZZ$-index theorem}\label{RmodZ}

Let \(\sigma_1,\sigma_2:\pi\to U(N)\) be unitary representations of the discrete group \(\pi\). Using the end-periodic eta invariant of MRS, we will define an end-periodic rho invariant \(\rho^\ep(\sigma_1,\sigma_2)\) analogous to the rho invariant in the APS case. This will determine a map from end-periodic \(K\)-homology to \(\mathbb{R}/\mathbb{Z}\), however we must be more careful about how we define the rho invariant due to the MRS index theorem not being applicable to all operators. 

\begin{definition} Let \((X,S,\gamma,f)\) be a \(K^\ep\)-cycle. Assume we can choose a covering function \(F:\tilde{X}\to\mathbb{R}\) so that the spectral sets of the families of the twisted operators \(D_1^+(X)\) and \(D_2^+(X)\) are discrete. Then we define the \emph{end-periodic rho invariant} to be \[\rho^\ep(\sigma_1,\sigma_2\,;X,S,\gamma,f)=\frac{1}{2}[h_1+\eta^\ep(D_1^+(X))-h_2-\eta^\ep(D_2^+(X))].\]\end{definition}

By Lemma 8.2 of \cite{MRS}, this definition is independent of the choice of such function \(F\), if it exists.

\begin{theorem}\label{thm:RmodZ}
Whenever it is defined, the mod\, \(\mathbb{Z}\) reduction of the end-periodic rho invariant \(\rho^\ep(\sigma_1,\sigma_2\,;X,S,\gamma,f)\) associated to \(\sigma_1,\sigma_2:\pi\to U(N)\) depends only on the representations \(\sigma_1,\sigma_2\) and the equivalence class of \((X,S,\gamma,f)\) in \(K_1^\ep(B\pi)\). Moreover, every equivalence class has a representative with a well-defined rho invariant. Hence there is a well-defined group homomorphism \[\rho^\ep(\sigma_1,\sigma_2):K_1^\ep(B\pi)\to\mathbb{R}/\mathbb{Z}.\] Furthermore, the following diagram commutes: \begin{center}\begin{tikzcd}[column sep=small]
K_1^\ep(B\pi) \arrow{dr}[swap]{\rho^\ep(\sigma_1,\sigma_2)} \arrow{rr}{\sim} & & \arrow{ll} K_1(B\pi) \arrow{dl}{\rho\,(\sigma_1,\sigma_2)} \\
&  \mathbb{R}/\mathbb{Z}   & 
\end{tikzcd}\end{center}
\end{theorem}

Hence, even if the spectral set of \(D^+(X)\) is not discrete, we can still define its \(\R/\Z\) end-periodic rho invariant in a perfectly reasonable and consistent manner. This allows us to define the \(\R/\Z\) invariant, for instance, in the case where \(\Ind D^+(X)\neq0\). For the applications to positive scalar curvature, the end-periodic rho invariant is well-defined and given by the usual formula \eqref{E:etapm}, since in Remark \ref{discrete_asssumptions} we have noted that the spectral sets of its twisted operators are discrete.

\begin{proof}
	That every equivalence class in \(K^\ep\)-homology has a representative  with discrete spectral set follows from the proof of Theorem \ref{eta1}---the cycle \((X,S,\gamma,f)\) is bordant to the cycle \((S^1\times Y,S^+\oplus S^+,d\theta,f)\), which has discrete spectral set by part (1) of Remark \ref{discrete_asssumptions}.
	
As we shall see, it is only necessary to prove invariance of \(\rho^\ep\) under bordism, and then Theorem \ref{eta1} will imply invariance under the other relations defining \(K^\ep\)-homology. First suppose that \((X,S,\gamma,f)\) is a boundary with Dirac operator \(D^+(X)\) such that the families associated to the twisted operators \(D_1^+(X)\) and \(D_2^+(X)\) have discrete spectral sets. We apply the MRS index theorem to each operator separately to get \[
\Ind_\text{MRS}D_i^+(Z_\infty)\;=\;\int_Z\; {\bf I}(D_i^+(Z)) - \int_Y\;\omega_i\, + 
\int_X\;dF\wedge \omega_i\, -\, \frac {h_i + \eta^\ep\,(D_i^+(X))}{2}
\] for \(i=1,2\). Now, since we are twisting by flat vector bundles, both the index form and the transgression classes for the twisted operators are constant multiplies of the index form and transgression class of the original operator. Hence when we subtract the two equations, the terms involving these vanish and we are left with \[\rho^\ep(\sigma_1,\sigma_2\,;X,S,\gamma,f)=\Ind_\text{MRS}D_2^+(Z_\infty)-\Ind_\text{MRS}D_1^+(Z_\infty)\] which is an integer. The end-periodic rho invariant behaves additively under disjoint unions of cycles and changes sign when the negative of a cycle is taken. This proves bordism invariance mod \(\mathbb{Z}\).

Now the \(K^\ep\)-cycle \((X,S,\gamma,f)\) with discrete spectral sets is bordant to \((S^1\times Y,S^+\oplus S^+,d\theta,f)\), where \(Y\) is Poincar{\'e} dual to \(\gamma\). By Section 6.3 of MRS \cite{MRS}, the end-periodic rho invariant of \((S^1\times Y,S^+\oplus S^+,d\theta,f)\) is equal to the rho invariant of the \(K\)-cycle \((Y,S^+,f)\). Hence \[\rho^\ep(\sigma_1,\sigma_2\,;X,S,\gamma,f)=\rho\,(\sigma_1,\sigma_2\,;Y,S^+,f) \mod\mathbb{Z}.\] The isomorphism \(K_1(B\pi)\cong K_1^\ep(B\pi)\) then immediately implies the theorem.
\end{proof}

\section{End-Periodic Bordism Groups}\label{end-per-B}

In this section, we recall the definition of the spin bordism groups, and introduce the analogous bordism groups in the end-periodic setting. As for \(K\)-homology, there are natural isomorphisms between the spin bordism groups and the end-periodic spin bordism groups. We also consider the PSC spin bordism groups described in Botvinnik-Gilkey \cite{BG}, and define the corresponding end-periodic PSC spin bordism groups. Throughout, we take \(m\geq5\) to be a positive odd integer.

\subsection{Spin bordism and end-periodic spin bordism}

We recall the definition of the spin bordism group \(\Omega_m^\spin(B\pi)\) for a discrete group \(\pi\). 

\begin{definition}
An \emph{\(\Omega_m^\spin\)-cycle} for \(B\pi\) is a triple \((M,\sigma,f)\), where \(M\) is a compact oriented  Riemannian spin manifold of dimension \(m\), \(\sigma\) is a choice of spin structure on \(M\), and \(f:M\to B\pi\) is a continuous map.
\end{definition} 

The \emph{negative} of an \(\Omega_m^\spin\)-cycle \((M,\sigma,f)\) is \((-M,\sigma,f)\), where \(-M\) is \(M\) with the reversed orientation. An \(\Omega_m^\spin\)-cycle \((M,\sigma,f)\) is a \emph{boundary} if there exists a compact oriented Riemannian manifold \(W\) with boundary \(\partial W=M\), a spin structure on \(W\) whose restriction to the boundary is the spin structure \(\sigma\), and a continuous map \(W\to B\pi\) extending the map \(f\). Two \(\Omega_m^\spin\)-cycles \((M,\sigma,f)\) and \((M',\sigma',f')\) are \emph{bordant} if \((M,\sigma,f)\amalg(-M',\sigma',f')\) is a boundary.

\begin{definition}
The \(m\)-dimensional spin bordism group \(\Omega_m^\spin(B\pi)\) for \(B\pi\), consists of \(\Omega_m^\spin\)-cycles for \(B\pi\) modulo the equivalence relation of bordism. It is an abelian group with addition given by disjoint union of cycles.
\end{definition}

The end-periodic spin bordism group \(\Omega_m^\eps(B\pi)\), is defined in an analogous way to the end-periodic \(K\)-homology group. 

\begin{definition} An \(\Omega_m^\eps\)-cycle for \(B\pi\) is a quadruple \((X,\sigma,\gamma,f)\) where \(X\) is a compact oriented Riemannian spin manifold of dimension \(m+1\), \(\sigma\) is a spin structure on \(X\), \(\gamma\) is a cohomology class in \(H^1(X,\mathbb{Z})\) that is primitive on each component of \(X\), and \(f:X\to B\pi\) is a continuous map. 
\end{definition} 

The definition of a boundary is essentially the same as for end-periodic \(K\)-homology.

\begin{definition}
An \(\Omega_m^\eps\)-cycle \((X,\sigma,\gamma,f)\) is a \emph{boundary} if there exists an end-periodic oriented Riemannian spin manifold \(Z_\infty\) with end modelled on \((X,\gamma)\), such that the pulled back spin structure \(\sigma\) on the periodic end extends to \(Z_\infty\), as does the pulled back map \(f\) to \(B\pi\).
\end{definition}

The \emph{negative} of a cycle \((X,\sigma,\gamma,f)\) is \((X,\sigma,-\gamma,f)\). As before, we introduce the additional relation of \emph{orientation/sign}: \[(X,\sigma,-\gamma,f)\sim(-X,\sigma,\gamma,f).\] Two \(\Omega_m^\eps\)-cycles \((X,\gamma,\sigma,f)\) and \((X',\gamma',\sigma',f')\) are \emph{bordant} if \((X,\sigma,\gamma,f)\amalg(X,\sigma,-\gamma,f)\) is a boundary.

\begin{definition}
The \(m\)-dimensional \emph{end-periodic spin bordism group} \(\Omega_m^\eps(B\pi)\) consists of \(\Omega_m^\eps\)-cycles modulo the equivalence relation generated by bordism and orientation/sign, with addition given by disjoint union.
\end{definition}

Analogous to the \(K\)-homology groups from Section 2, there is a canonical isomorphism between the spin bordism and end-periodic spin bordism groups which we will now describe. 

The map \(\Omega_m^{\spin}(B\pi)\to\Omega_m^{\eps}(B\pi)\) takes a \(\Omega_m^{\spin}(B\pi)\)-cycle \((M,\sigma,f)\) to \((S^1\times M,1\times\sigma,d\theta,f)\), where \(S^1\times M\) has the product orientation and Riemannian metric, \(1\times\sigma\) is the product spin structure of the trivial spin structure \(1\) on \(S^1\) with the spin structure \(\sigma\) on \(M\), \(d\theta\) is the standard generator of the first cohomology of \(S^1\), and \(f\) is the obvious extension of \(f:M\to B\pi\) to \(S^1\times M\).

\begin{proposition}
The map which sends an \(\Omega_m^{\spin}(B\pi)\)-cycle \((M,\sigma,f)\) to the \(\Omega_m^{\eps}(B\pi)\)-cycle \((S^1\times M,1\times\sigma,d\theta,f)\) is well-defined on spin bordism groups.
\end{proposition}

\begin{proof}
If \((M,\sigma,f)\) and \((M',\sigma',f')\) are bordant, with \(W\) bounding their disjoint union, then \(\mathbb{R}_{\geq0}\times M\) and \(\mathbb{R}_{\leq0}\times M'\) can be joined using \(W\) to form and end-periodic manifold \(Z_\infty\) with multiple ends. All structures extend to \(Z_\infty\) by assumption, hence the two \(\Omega_m^{\eps}(B\pi)\)-cycles \((S^1\times M,1\times\sigma,d\theta,f)\) and \((-S^1\times M,1\times\sigma',-d\theta,f')\) are bordant. Using the orientation/sign relation, we see that \((S^1\times M,1\times\sigma,d\theta,f)\) and \((S^1\times M',1\times\sigma',d\theta,f')\) are equivalent.
\end{proof}

Now for the map \(\Omega_m^{\eps}(B\pi)\to\Omega_m^{\spin}(B\pi)\). Let \((X,\sigma,\gamma,f)\) be an \(\Omega_m^\eps\)-cycle for \(B\pi\), and \(Y\) be a submanifold of \(X\) Poincar{\'e} dual to \(\gamma\). We equip \(Y\) with the induced spin structure and orientation from \(\gamma\). Explicitly, the orientation of \(Y\) is as in the paragraph after Definition \ref{ep-isom}, and the restricted spin structure is obtained first by cutting \(X\) open along \(Y\) to get a manifold \(W\) with boundary \(\partial W=Y\amalg-Y\), and then taking the boundary spin structure on the positively oriented component \(Y\) of \(\partial W\).  This yields an \(\Omega_m^\spin\)-cycle \((Y,\sigma,f)\), where \(\sigma\) and \(f\) are restricted to \(Y\).

\begin{proposition}
The map taking an \(\Omega_m^{\eps}(B\pi)\)-cycle \((X,\sigma,\gamma,f)\) to the \(\Omega_m^{\spin}(B\pi)\)-cycle \((Y,\sigma,f)\) described above is well-defined on bordism groups.
\end{proposition}

\begin{proof}
Independence of the choice of \(Y\) is proved as for the \(K\)-homology case, only with spin structures instead of Dirac bundles. It is clear that the orientation/sign relation is respected, since both \((X,\sigma,-\gamma,f)\) and \((-X,\sigma,\gamma,f)\) get sent to \((-Y,\sigma,f)\). If \((X,\sigma,\gamma,f)\) and \((X',\sigma',\gamma',f')\) are bordant, then there is a compact manifold \(Z\) with boundary \(\partial Z=Y\amalg-Y'\) such that the spin structures and maps extend over \(Z\). But this shows that \((Y,\sigma,f)\) and \((Y',\sigma',f')\) are bordant.
\end{proof}

\begin{theorem}\label{thm:bor-isom}
The above maps of bordism groups are inverse to each other, and so define a natural isomorphism of abelian groups \(\Omega_m^\spin(B\pi)\cong \Omega_m^\eps(B\pi)\).
\end{theorem}

\begin{proof}
A cycle \((M,\sigma,f)\) gets mapped to \((S^1\times M,1\times\sigma,d\theta,f)\), which gets returned to \((M,1\times\sigma,f)\), where the latter two entries are restricted to \(M\). It is straightforward to check that the product spin structure \(1\times \sigma\) restricted to \(M\) yields the original spin structure \(\sigma\). Therefore we obtain our original cycle \((M,\sigma,f)\) after mapping it to and from end-periodic bordism.

Now let \((X,\sigma,\gamma,f)\) be an end-periodic cycle, with submanifold \(Y\) Poincar{\'e} dual to \(\gamma\). This maps to a cycle \((Y,\sigma,f)\), where the latter two structures are restricted from \(X\), and this maps back to \((S^1\times Y,1\times\sigma,d\theta,f)\). The same argument as in the proof of Theorem \ref{ep-isom} shows that this is bordant to \((X,\sigma,\gamma,f)\).
\end{proof}

%Take a gluing cocycle \(g_{ij}\) for the principal \(SO(m)\)-bundle \(P_{SO(m)}(M)\to M\). The spin structure on \(M\) is represented by a lift \(\tilde{g}_{ij}\) of this cocycle to \(\text{Spin}(m)\). The principal \(SO(1)\) bundle on \(S^1\) is represented by the gluing cocycle \(h_{ij}=1\in SO(1)\), and the trivial spin structure by the lift \(\tilde{h}_{ij}=1\in\text{Spin}(1)=\{\pm1\}\). For any \(k\) and \(l\) there are canonical embeddings \(SO(k)\times SO(l)\to SO(k+l)\) and \(\text{Spin}(k)\times\text{Spin}(l)\to\text{Spin}(k+l)\), and a commuting diagram \begin{center}\begin{tikzcd}[column sep=small]
%\text{Spin}(k)\times\text{Spin}(l) \arrow{d} \arrow{r} & \text{Spin}(k+l) \arrow{d} \\
%SO(k)\times SO(l) \arrow{r}   & SO(k+l).
%\end{tikzcd}\end{center} The \(SO(m+1)\) frame bundle on \(S^1\times M\) can be represented by the cocycle \(g_{ij}\oplus h_{ij}\), and our lift of these transition functions is then given by \(\tilde{g}_{ij}\oplus\tilde{h}_{ij}\). There is an embedding of \(P_{SO}(M)\) into \(P_{SO}(S^1\times M)\) given by \((e_1,\ldots,e_m)\mapsto(\partial_\theta,e_1,\ldots,e_m)\). The spin structure on \(M\) is obtained by restricting the map \(P_\text{Spin}(S^1\times M)\to P_{SO}(S^1\times M)\) to the preimage of \(P_{SO}(M)\).

\subsection{PSC spin bordism and end-periodic PSC spin bordism}

In \cite{BG}, Botvinnik and Gilkey use a variant of spin cobodism tailored to the setting of manifolds with positive scalar curvature, which we now recall.

\begin{definition} 
A \(\Omega_m^{\spin,+}\)-cycle is a quadruple \((M,g,\sigma,f)\), where \(M\) is a compact oriented Riemannian spin manifold of dimension \(m\) with a metric \(g\) of positive scalar curvature, \(\sigma\) is a spin structure on \(M\), and \(f:M\to B\pi\) is a continuous map. 
\end{definition} 

The negative of \((M,g,\sigma,f)\) is \((-M,g,\sigma,f)\), as before. A cycle \((M,g,\sigma,f)\) is called a \emph{boundary} if there is a compact oriented Riemannian spin manifold \(W\) with boundary \(\partial W=M\) so that the spin structure \(\sigma\) and map \(f\) extend to \(W\). It is also required that \(W\) has a metric of positive scalar curvature that is a product metric \(dt^2+g\) in a neighbourhood of the boundary. Two cycles are \emph{bordant} if the disjoint union of one with the negative of the other is a boundary.

\begin{definition}
The \emph{PSC spin bordism group} \(\Omega_m^{\spin,+}(B\pi)\) for \(B\pi\) consists of \(\Omega_m^{\spin,+}\)-cycles modulo bordism, with addition given by disjoint union.
\end{definition}

We now define the end-periodic PSC spin bordism group \(\Omega_m^{\ep,\spin,+}(B\pi)\) for \(B\pi\). 

\begin{definition}
An \emph{\(\Omega_m^{\ep,\spin,+}\)-cycle} is a quintuple \((X,g,\sigma,\gamma,f)\), where \(X\) is a compact oriented Riemannian spin manifold of dimension \(m+1\) with a metric \(g\) of positive scalar curvature, \(\sigma\) is a choice of spin structure on \(X\), \(\gamma\) is a cohomology class in \(H^1(X,\mathbb{Z})\) whose restriction to each component of \(X\) is primitive, and \(f:X\to B\pi\) is a continuous map. We further require that there is a submanifold \(Y\) of \(X\) that is Poincar{\'e} dual to \(\gamma\), such that the induced metric on \(Y\) has positive scalar curvature, and the metric on \(X\) is a product metric \(dt^2+g_Y\) in a neighbourhood of \(Y\).
\end{definition}

Let \((X,g,\sigma,\gamma,f)\) be an \(\Omega_m^{\ep,\spin,+}\)-cycle and take \(Y\subset X\) to be a submanifold with PSC that is Poincar{\'e} dual to \(\gamma\), having the product metric in a tubular neighbourhood. As before we form \(X_1=\bigcup_{k\geq0}W_k\), where the \(W_k\) are isometric copies of \(X\) cut open along \(Y\). For \((X,g,\sigma,\gamma,f)\) to be a \emph{boundary} means that there is a compact oriented Riemannian spin manifold \(Z\) of positive scalar, whose metric is a product near the boundary, which can be attached to \(X_1\) along \(Y\) to form a complete oriented Riemannian spin manifold of PSC \(Z_\infty=Z\cup_YX_1\), such that the pulled back spin structure \(\sigma\) and map \(f\) on \(X_1\) extend over \(Z\).

The \emph{negative} of \((X,g,\sigma,\gamma,f)\) is \((X,g,\sigma,-\gamma,f)\), and we have the \emph{orientation/sign} relation \[(X,g,\sigma,-\gamma,f)\sim(-X,g,\sigma,\gamma,f).\] Two \(\Omega_m^{\eps,+}\)-cycles are \emph{bordant} if the disjoint union of one with the negative of the other is a boundary.

\begin{definition}
The \(m\)-dimensional \emph{end-periodic PSC spin bordism group} \(\Omega_m^{\ep,\spin,+}(B\pi)\) for \(B\pi\) consists of \(\Omega_m^{\eps,+}\)-cycles modulo bordism and orientation/sign, with addition given by disjoint union.
\end{definition}

%We will now describe a canonical isomorphism between the two PSC spin bordism theories. 
 %The only subtlety in checking that these maps are well-defined is that when we cut \(X\) open along \(Y\) to get a manifold \(W\) of PSC with boundary \(\partial W=Y\amalg-Y\), the metric is not necessarily a product near the boundary. However, since \(Y\) has PSC, and we are also assuming the dimension \(m\geq5\) is sufficiently large, we can modify the metric in a neighbourhood of the boundary to make it a product without losing the positive scalar curvature.

\begin{theorem}\label{thm:psc-isom}
There is a canonical isomorphism \(\Omega_m^{\spin,+}(B\pi)\cong\Omega_m^{\eps,+}(B\pi)\).
\end{theorem}

The maps are exactly as for the spin bordism theories, only when mapping from \(\Omega_m^{\ep,\spin,+}(B\pi)\) to \(\Omega_m^{\spin,+}(B\pi)\) the Poincar{\'e} dual submanifold \(Y\) must be taken to have PSC and a product metric in a tubular neighbourhood.

\begin{proof}
As before.
\end{proof}

\subsection{Rho invariants}

Given a triple \((M,\sigma,f)\) and two unitary representations \(\sigma_1,\sigma_2:\pi\to U(N)\), we define the rho invariant \(\rho\,(\sigma_1,\sigma_2\,;M,\sigma,f)\) as before, using the spin Dirac operator for the cycle \((M,S,f)\). We also define the end-periodic rho invariant for cycles \((X,\sigma,\gamma,f)\) in an entirely analogous manner, using the end-periodic eta invariant of MRS instead. Of course, we must again be careful with the definition, allowing only the rho invariant for cycles whose twisted operators have discrete spectral sets to be defined in terms of the true end-periodic eta invariants---all others are defined by taking bordant cycles with discrete spectra. We remark also that in the case of positive scalar curvature, the \(h\)-terms appearing in the definition of the rho invariants vanish.

\begin{theorem}
The rho invariant extends to a well-defined homomorphism \[\rho\,(\sigma_1,\sigma_2):\Omega_m^\spin(B\pi)\to\mathbb{R}/\mathbb{Z},\] as does the end-periodic rho invariant \[\rho^\ep(\sigma_1,\sigma_2):\Omega_m^\eps(B\pi)\to\mathbb{R}/\mathbb{Z}.\] Furthermore, the following diagram commutes: \begin{center}\begin{tikzcd}[column sep=small]
\Omega_m^\eps(B\pi) \arrow{dr}[swap]{\rho^\ep(\sigma_1,\sigma_2)} \arrow{rr}{\sim} & & \arrow{ll} \Omega_m^\spin(B\pi) \arrow{dl}{\rho\,(\sigma_1,\sigma_2)} \\
&  \mathbb{R}/\mathbb{Z}   & 
\end{tikzcd}\end{center}
\end{theorem}

\begin{proof}
Apply the APS and MRS index theorems respectively, and use the isomorphism of Theorem \ref{thm:bor-isom}.
\end{proof}

Now for the positive scalar curvature case.

\begin{theorem}\label{thm:Rmaps}
The rho invariant extends to a well-defined homomorphism \[\rho\,(\sigma_1,\sigma_2):\Omega_m^{\spin,+}(B\pi)\to\mathbb{R},\] as does the end-periodic rho invariant \[\rho^\ep(\sigma_1,\sigma_2):\Omega_m^{\eps,+}(B\pi)\to\mathbb{R}.\] Furthermore, the following diagram commutes: \begin{center}\begin{tikzcd}[column sep=small]
\Omega_m^{\eps,+}(B\pi) \arrow{dr}[swap]{\rho^\ep(\sigma_1,\sigma_2)} \arrow{rr}{\sim} & & \arrow{ll} \Omega_m^{\spin,+}(B\pi) \arrow{dl}{\rho\,(\sigma_1,\sigma_2)} \\
&  \mathbb{R}  & 
\end{tikzcd}\end{center}
\end{theorem}

\begin{remark}
	The end-periodic rho invariant appearing in the theorem is given on all representatives of equivalence classes as the genuine difference of the twisted eta invariants as in formula \eqref{E:etapm}, due to Remark \ref{discrete_asssumptions}.
\end{remark}

For the proof, we will need the following (cf. \cite{MRS}, Proposition 8.5 (ii)).

\begin{lemma}\label{lem:ep-eq}
If \((X,g,\sigma,\gamma,f)\) is an \(\Omega_m^{\eps,+}\)-cycle and \((Y,g,\sigma,f)\) is the \(\Omega_m^{\spin,+}\)-cycle it maps to, then \[\rho^\ep(\sigma_1,\sigma_2\,;X,g,\sigma,\gamma,f)=\rho(\sigma_1,\sigma_2\,;Y,g,\sigma,f)\]
\end{lemma}

\begin{proof}
We join \(\mathbb{R}_{\geq0}\times Y\) to \(X_1=\cup_{k\geq0}W_k\) together as in Figure \ref{fig:epc2} to form an end-periodic spin manifold \(Z_\infty\) with two ends. Lemma 8.1 of \cite{MRS} (which uses the results of Gromov-Lawson \cite{GrLaw3}) gives that the spin Dirac operator \(D^+(Z_\infty)\) is Fredholm and has zero index. The same holds for its twisted counterparts. Applying the MRS index theorem to the two twisted spin Dirac operators \(D_1^+(Z_\infty)\) and \(D_2^+(Z_\infty)\), and subtracting the equations as per usual then yields the result.
\end{proof}

\begin{proof}[Proof of Theorem \ref{thm:Rmaps}]
See Theorem 1.1 of Botvinnik-Gilkey \cite{BG} for the proof that the map \(\rho\,(\sigma_1,\sigma_2):\Omega_m^{\spin,+}(B\pi)\to\mathbb{R}\) is well-defined. Lemma \ref{lem:ep-eq} and the isomorphism of Theorem \ref{thm:psc-isom} then immediately imply the result.
\end{proof}

\section{End-periodic structure group}\label{structure}

Let \(\sigma_1,\sigma_2:\pi\to U(N)\) be unitary representations of the discrete group \(\pi\). 
Recall the definition of the structure group $S_1(\sigma_1, \sigma_2)$ of Higson-Roe, starting from Definition 8.7 of \cite{HigsonRoe}.

\begin{definition} An \emph{odd $(\sigma_1, \sigma_2)$-cycle} is a quintuple $(M, S, f, D, n)$ where $(M, S, f)$ is an odd \(K\)-cycle for $B\pi$, $D$ is a Dirac operator for $(M, S, f)$, and $n\in \ZZ$.
\end{definition}

A  $(\sigma_1, \sigma_2)$-cycle $(M, S, f, D, n)$ is a {\em boundary} if the \(K\)-cycle \((M,S,f)\) is a bounded by a manifold \(W\) (as in Definition \ref{boundary}) and there are Dirac operators \(D_1(W)\) and \(D_2(W)\) on \(W\) which bound the twisted Dirac operators \(D_1\) and \(D_2\) on \(M\), such that
$$
\Ind_\text{APS}D_1^+(W) - \Ind_\text{APS}D_2^+(W) = n.
$$

Since we are no longer looking at rho invariants modulo integers or at spin Dirac operators, we will denote by \(\rho(\sigma_1,\sigma_2\,;D,f)\) the rho invariant of definition \ref{def:rho}, indicating its possible dependence on the Dirac operator \(D\).

\begin{lemma}[\cite{HigsonRoe} Lemma 8.10]\label{HR-lemma} If a  $(\sigma_1, \sigma_2)$-cycle $(M, S, f, D, n)$ is a boundary, then
$\rho(\sigma_1, \sigma_2\,;D, f) + n = 0.$
\end{lemma}

\begin{definition}  The \emph{relative eta invariant}, or \emph{rho invariant} of the $(\sigma_1, \sigma_2)$-cycle $(M, S, f, D, n)$ is $\rho(\sigma_1, \sigma_2\,;D,f) + n$.
\end{definition}

The {\em disjoint union} of  $(\sigma_1, \sigma_2)$-cycles is defined as,
$$
(M, S, f, D, n) \amalg (M', S', f', D', n')
=(M \amalg M',S \amalg S',f \amalg f',D \amalg D',n +n').
$$ The {\em negative} of a $(\sigma_1, \sigma_2)$-cycle $(M,S,f,D,n)$, is defined as,
$$
-(M,S,f,D,n)=(M,-S,f,-D,h_1 - h_2 -n), 
$$
where $h_1 = \dim \ker(D_1)$ and $h_2 = \dim \ker(D_2)$. Two $(\sigma_1, \sigma_2)$-cycles are {\em bordant} if the disjoint union of one cycle with the negative of the other is a boundary. 

The two remaining relations to define are: \begin{itemize}\setlength\itemsep{2mm}
\item \emph{Direct sum/disjoint union}: \[(M, S \oplus S' , f, D \oplus D' , n)\sim(M\amalg M,S\amalg S',f\amalg f,D\amalg D',n).\]
\item \emph{Bundle Modification}: If $(\hat{M}, \hat{S}, \hat{f})$ is an elementary bundle modification of $(M, S, f)$
with the Dirac operator $\hat{D}$ from \ref{rem:tensorDirac},
then $(M, S, f, D, n)\sim(\hat M, \hat S,\hat f,\hat D, n)$.
\end{itemize}

\begin{definition} The  structure group $S(\sigma_1, \sigma_2)$, is the set of equivalence classes
of $(\sigma_1, \sigma_2)$-cycles under the equivalence relation generated by bordism, direct sum/disjoint union, and bundle modification. It is an abelian group with addition is given by disjoint union.
\end{definition}

In \cite{HigsonRoe} Proposition 8.14, it is proved that the relative eta invariant of a $(\sigma_1, \sigma_2)$-cycle depends only on the class that the cycle determines in $S(\sigma_1, \sigma_2)$. Hence there is a well-defined group homomorphism $ \rho: S(\sigma_1, \sigma_2) \to \RE$, defined by
$$
 \rho(M, S, f, D, n)=\rho(\sigma_1,\sigma_2\,;D,f) + n.
$$

\subsection{End-periodic structure group}

We define in a parallel manner the end-periodic structure group $S_1^\ep(\sigma_1, \sigma_2)$.

\begin{definition}An \emph{odd $(\sigma_1, \sigma_2)^\ep$-cycle} is a sextuple $(X, S, \gamma, f, D, n)$ where $(X, S, \gamma, f)$ is a $K^\ep$-cycle for $B\pi$, $D$ is a Dirac operator for $(X, S,\gamma, f)$, and $n\in \ZZ$. We additionally assume that the spectral set of the family \(D_z^+(X)\) is discrete.
\end{definition}

A  $(\sigma_1, \sigma_2)^\ep$-cycle $(X, S, f,\gamma, D, n)$ is a {\em boundary} 
if the \(K^\ep\)-cycle \((X,S,\gamma,f)\) is a boundary (Definition \ref{ep-boundary}), and moreover there is a Dirac operator $D(Z_\infty)$ on the manifold $Z_\infty$ extending the Dirac operator $D$ on $X_1=\bigcup_{k\geq0}W_k$ such that
the difference of the MRS indices
$$
\Ind_\text{MRS}(D_1^+(Z_\infty)) - \Ind_\text{MRS}(D_2^+(Z_\infty)) = n.
$$
Here the \(D_i^+(Z_\infty)\) are the twists of \(D^+(Z_\infty)\) by the flat vector bundles determined by the extension of \(f\) to \(Z_\infty\) and by \(\sigma_1,\sigma_2\).
We can show the analog of Lemma \ref{HR-lemma}

 \begin{lemma}If a  $(\sigma_1, \sigma_2)^\ep$-cycle $(X, S, \gamma, f, D, n)$ is a boundary, then
$\rho^\ep(\sigma_1, \sigma_2\,;D, f, \gamma) + n = 0.$
\end{lemma}

We call the quantity $\rho^\ep(\sigma_1,\sigma_2\,;D, f, \gamma) + n $ the \emph{end-periodic rho invariant} of the 
$(\sigma_1, \sigma_2)^\ep$-cycle $(X, S, \gamma, f, D, n)$.\bigskip

The {\em disjoint union} of  $(\sigma_1, \sigma_2)^\ep$-cycles is defined as
$$
(X, S, f, \gamma, D, n) \amalg (X', S', \gamma', f', D', n')
=(X \amalg X',S \amalg S',\gamma \amalg \gamma',f \amalg f',D \amalg D',n +n').
$$ The {\em negative} of a $(\sigma_1, \sigma_2)^\ep$-cycle $(X,S,\gamma, f,D,n)$, is
$$
-(X,S,\gamma, f,D,n)=(X,S,-\gamma,f,D,h_1 - h_2 -n), 
$$
where $h_1, h_2$ are the integers occurring in the MRS index theorem associated to $\sigma_1, \sigma_2$. Two $(\sigma_1, \sigma_2)^\ep$-cycles are {\em bordant} if the disjoint union of one with the negative of the other is a boundary. We also have: 
\begin{itemize}
\setlength\itemsep{2mm}

\item \emph{Direct sum/disjoint union}: \[(X, S \oplus S' , \gamma+\gamma' f, D \oplus D' , n)\sim(X\amalg M,S\amalg S', \gamma \amalg \gamma',f\amalg f,D\amalg D',n).\]

\item \emph{Bundle Modification}: If $(\hat X, \hat S,\hat \gamma, \hat f)$ is an elementary bundle modification of $(X, S, \gamma, f)$ and \(\hat D\) is the Dirac operator of Remark \ref{rem:tensorDirac},
then $(X, S, \gamma, f, D, n)\sim(\hat X, \hat S,\hat \gamma, \hat f,\hat D, n)$.

\item \emph{Orientation/sign}: \[(X,S,-\gamma,f,D,n)\sim(-X,\Pi(S),\gamma,f,D,n).\]
\end{itemize}

\begin{definition} The  end-periodic structure group, denoted by $S_1^\ep(\sigma_1, \sigma_2)$, is the set of equivalence classes
of $(\sigma_1, \sigma_2)^\ep$-cycles under the equivalence relation generated by bordism, direct sum/disjoint union, bundle modification, and orientation/sign. It is is an abelian group with unit
and addition is given by disjoint union.
\end{definition}

Define the group homomorphism $ \rho^\ep: S_1^\ep(\sigma_1, \sigma_2) \to \RE$
by the formula,
$$
 \rho^\ep(X, S,\gamma, f, D, n)=\rho^\ep(\sigma_1, \sigma_2 \,; D, f,\gamma ) + n.
$$ Then the following theorem is the analog of Theorem \ref{thm:RmodZ} is is proved in a similar way.

\begin{theorem}\label{thm:R}
The end-periodic rho invariant \(\rho^\ep(X,S,\gamma,f, \sigma_1, \sigma_2) +n\) associated to 
the $(\sigma_1, \sigma_2)^\ep$-cycle $(M,S,\gamma, f,D,n)$
depends only on the equivalence class of \((M,S,\gamma, f,D,n)\) in \(S_1^\ep(\sigma_1, \sigma_2)\). 
Hence there is a well-defined group homomorphism \[\rho^\ep:S_1^\ep(\sigma_1, \sigma_2)\to\mathbb{R}.\] 
Furthermore, the following diagram commutes: \begin{center}\begin{tikzcd}[column sep=small]
S_1^\ep(\sigma_1, \sigma_2) \arrow{dr}[swap]{\rho^\ep} \arrow{rr}{\sim} & & \arrow{ll} S_1(\sigma_1, \sigma_2) \arrow{dl}{\rho} \\
&  \mathbb{R}  & 
\end{tikzcd}\end{center}
\end{theorem}

Here the maps $S_1^\ep(\sigma_1, \sigma_2)  \leftrightarrow S_1(\sigma_1, \sigma_2) $ are the analog of the maps in \(K\)-homologies
given earlier.

Also, Higson-Roe establish a commuting diagram of short exact sequences, cf. \cite{HigsonRoe} the paragraph below Definition 8.6,

\begin{equation}
\begin{gathered}
\xymatrix{
0\ar[r] &\ZZ\ar[r]\ar[d]^=\ar[r]& S_1(\sigma_1, \sigma_2)\ar[r]\ar[d]^\rho & K_1(B\pi) \ar[d]^{\rho(\sigma_1,\sigma_2)}\ar[r] & 0\\
0\ar[r] &\ZZ \ar[r]\ar[r]& \RE\ar[r]  & \RE/\ZZ  \ar[r] & 0.
}
\end{gathered}
\end{equation}

By Theorems \ref{thm:R} and \ref{thm:RmodZ}, we deduce that there is a commuting diagram of short exact sequences, 

\begin{equation}
\begin{gathered}
\xymatrix{
0\ar[r] &\ZZ\ar[r]\ar[d]^=\ar[r]& S_1^\ep(\sigma_1, \sigma_2)\ar[r]\ar[d]^{\rho^\ep}& K_1^\ep(B\pi) \ar[d]^{\rho^\ep(\sigma_1,\sigma_2)}\ar[r] & 0\\\
0\ar[r] &\ZZ \ar[r]\ar[r]& \RE\ar[r]  & \RE/\ZZ  \ar[r] & 0.\
}
\end{gathered}
\end{equation}
This tells us when the $\RE/\ZZ$-index theorem can be refined to an $\RE$-index theorem.

%%%%%%%%%%%%%%%%%%%%
\section{Applications to positive scalar curvature}\label{PSC}
%%%%%%%%%%%%%%%%%%%%%

Using the above isomorphisms of \(K\)-homologies and cobordism theories, we can immediately transfer results on positive scalar curvature from the odd-dimensional case to the even-dimensional case in which a primitive 1-form is given.

\subsection{Odd-dimensional results in the literature}

First we will state the odd-dimensional results that we will be generalising to the even-dimensional case using our isomorphisms. The first ones are obstructions to positive scalar curvature.
%due to Higson and Roe -- Theorem 1.1 from \cite{HigsonRoe}.

\begin{theorem}[Weinberger \cite{Weinberger}, Higson-Roe Theorem 6.9 \cite{HigsonRoe}]\label{HRW-psc}
Let \((M,S,f)\) be an odd \(K\)-cycle for \(B\pi\), where \(M\) is an odd dimensional spin manifold with a Riemannian metric of positive scalar curvature, and \(S\) is the bundle of spinors on $M$. 
%If the maximal Baum-Connes conjecture holds for \(\pi\), 
Then for any pair of unitary representations \(\sigma_1,\sigma_2:\pi\to U(N)\), the associated rho invariant \(\rho(\sigma_1,\sigma_2\,;M,S,f)\) is a rational number.
\end{theorem}

\begin{theorem}[Higson-Roe Remark 6.10 \cite{HigsonRoe}]\label{HR-psc}
Let \((M,S,f)\) be an odd \(K\)-cycle for \(B\pi\), where \(M\) is an odd dimensional spin manifold with a Riemannian metric of positive scalar curvature, and \(S\) is the bundle of spinors on $M$. 
If the maximal Baum-Connes map for \(\pi\) is injective, 
then for any pair of unitary representations \(\sigma_1,\sigma_2:\pi\to U(N)\), the associated rho invariant \(\rho(\sigma_1,\sigma_2\,;M,S,f)\) is an integer.
\end{theorem}

\begin{remarks}
The maximal Baum-Connes map for \(\pi\) is injective whenever for instance $\pi$ is a torsion-free linear discrete group, \cite{GHW}.
\end{remarks}

\begin{theorem}[Higson-Roe Theorem 1.1 \cite{HigsonRoe}, Keswani \cite{Keswani00}]\label{HRK-psc}
Let \((M,S,f)\) be an odd \(K\)-cycle for \(B\pi\), where \(M\) is an odd dimensional spin manifold with a Riemannian metric of positive scalar curvature, and \(S\) is the bundle of spinors on $M$. 
If the maximal Baum-Connes conjecture holds for \(\pi\), 
then for any pair of unitary representations \(\sigma_1,\sigma_2:\pi\to U(N)\), the associated rho invariant \(\rho(\sigma_1,\sigma_2\,;M,S,f)\) is zero.
\end{theorem}

\begin{remarks}
The maximal Baum-Connes conjecture holds for $\pi$ whenever $\pi$ is \(K\)-amenable.
\end{remarks}

We now turn to a result on the number of path components of the moduli space of PSC metrics modulo diffeomorphism, \(\fM^+(M)\). Denote for a group \(\pi\), the representation ring \(R(\pi)\) consisting of formal differences of finite dimensional unitary representations, and let \(R_0(\pi)\) be those formal differences with virtual dimension zero (an element of \(R_0(\pi)\) can be thought of as an ordered pair of unitary representations \(\sigma_1,\sigma_2:\pi\to U(N)\)). Following Botvinnik and Gilkey \cite{BG}, introduce the subgroups \[R_0^\pm(\pi)=\{\alpha\in R_0(\pi):\tr(\alpha(\lambda))=\pm\tr(\alpha(\lambda^{-1}))\text{ for all } \lambda\in\pi\}\] and define \[r_m(\pi)=\begin{cases}
\rank_\mathbb{Z}R_0^+(\pi) & \text{ if }\quad m=3 \mod 4, \\
\rank_\mathbb{Z}R_0^-(\pi) & \text{ if }\quad m=1 \mod 4 .
 \end{cases}\] The following is a result of Botvinnik and Gilkey on the number of path components of the moduli space of PSC metrics modulo diffeomorphism.

\begin{theorem}[Botvinnik-Gilkey Theorem 0.3 \cite{BG}]\label{thm:BG}
Let \(M\) be a compact connected spin manifold of odd dimension \(m\geq5\) admitting a metric of positive scalar curvature. Suppose that \(\pi=\pi_1(M)\) is finite and nontrivial, and that \(r_m(\pi)>0\). Then the moduli space of PSC metrics modulo diffeomorphism \(\fM^+(M)\) has infinitely many path components.
\end{theorem}

Their proof involves finding a countably indexed family of metrics \(g_i\) of positive scalar curvature on \(M\) so that \(\rho(M,g_i)\neq\rho(M,g_j)\) for \(i\neq j\). If these metrics were homotopic through PSC metrics, then they would lie in the same PSC bordism class and hence have equal rho invariants. We will extend this result to the even-dimensional case under the additional hypothesis of `psc-adaptability'; see Definition \ref{def:psc-ad}.

\subsection{Our even dimensional results}

In the following theorems, we assume that $Y$ is a submanifold of $X$ that is  Poincar\'e dual  to a primitive class $\gamma \in  H^1(X, \ZZ)$ 
such that the scalar curvature of $Y$ in the induced metric is positive. By a theorem of \cite{SY}, if $\dim(X)=n\le 7$, then every homology class in $H_{n-1}(X, \ZZ)$
has a representative that is a smooth, orientable minimal hypersurface. It follows that if $X$ is spin with positive scalar curvature, 
then Poincar\'e dual  to a primitive class $\gamma \in  H^1(X, \ZZ)$
can be chosen to be a smooth, spin minimal hypersurface $Y$, and it follows that the scalar curvature of $Y$ in the induced metric is positive. So our assumption in the Theorems below are automatically true when $\dim(X)=n\le 7$.

The following is our even dimensional analog of Theorem \ref{HRW-psc}.

\begin{theorem}\label{thm:ev}
Let \((X,S,\gamma, f)\) be an odd \(K^\ep\)-cycle for \(B\pi\), where \(X\) is an even dimensional spin manifold with a Riemannian metric of positive scalar curvature, \(S\) is the bundle of spinors on $X$ and $\gamma$ a primitive class in $H^1(X, \ZZ)$ such that there is a Poincar\'e dual submanifold $Y$
whose scalar curvature in the induced metric is positive. 
Then for any pair of unitary representations \(\sigma_1,\sigma_2:\pi\to U(N)\), the associated end-periodic rho invariant \(\rho^\ep(\sigma_1, \sigma_2\,;X,S,\gamma, f)\) is a rational number.
\end{theorem}

\begin{proof} 
The odd \(K^\ep\)-cycle for \(B\pi\),  \((X,S,\gamma, f)\) determines an odd \(K\)-cycle for \(B\pi\),  \((Y,S^+,f)\) where $Y$
is a Poincar\'e dual submanifold for $\gamma$ having positive scalar curvature, and is given the induced spin structure from \(X\). 
By Theorem \ref{HRW-psc}, $\rho(\sigma_1, \sigma_2\,;Y,S^+,f)\in \QQ$. By  Theorem \ref{thm:RmodZ} it follows that 
\(\rho^\ep(\sigma_1, \sigma_2\,;X,S,\gamma, f) \in \QQ\) as claimed.
\end{proof}

Next is our even dimensional analog of Theorem \ref{HR-psc}, and is argued as above.

\begin{theorem}
Let \((X,S,\gamma, f)\) be an odd \(K^\ep\)-cycle for \(B\pi\), where \(X\) is an even dimensional spin manifold with a Riemannian metric of positive scalar curvature, \(S\) is the bundle of spinors on $X$ and $\gamma$ a primitive class in $H^1(X, \ZZ)$ such that there is a Poincar\'e dual submanifold $Y$
whose scalar curvature in the induced metric is positive. 
If the maximal Baum-Connes map for \(\pi\) is injective, then for any pair of unitary representations \(\sigma_1,\sigma_2:\pi\to U(N)\), the associated end-periodic rho invariant \(\rho^\ep(\sigma_1, \sigma_2\,;X,S,\gamma, f)\) is an integer.
\end{theorem}

\begin{proof} 
As for Theorem \ref{thm:ev}.
\end{proof}

Here is the even dimensional analog of Theorem \ref{HRK-psc}.

\begin{theorem}
Let \((X,S,\gamma, f)\) be an odd \(K^\ep\)-cycle for \(B\pi\), where \(X\) is an even dimensional spin manifold with a Riemannian metric of positive scalar curvature, \(S\) is the bundle of spinors on $X$ and $\gamma$ a primitive class in $H^1(X, \ZZ)$ such that there is a Poincar\'e dual submanifold $Y$
whose scalar curvature in the induced metric is positive. 
If the maximal Baum-Connes conjecture holds for \(\pi\), then for any pair of unitary representations \(\sigma_1,\sigma_2:\pi\to U(N)\), the associated end-periodic rho invariant \(\rho^\ep(\sigma_1, \sigma_2\,;X,S,\gamma, f)\) is zero.
\end{theorem}

\begin{proof} 
The odd \(K^\ep\)-cycle for \(B\pi\),  \((X,S,\gamma, f)\) determines an odd \(K\)-cycle  \((Y,S^+,f)\) for \(B\pi\),  where $Y$
is a Poincar\'e dual submanifold for $\gamma$ having positive scalar curvature, and is endowed with the induced spin structure.
By Theorem \ref{HRK-psc}, $\rho( \sigma_1, \sigma_2\,;Y,S^+,f)=0$. By \ref{lem:ep-eq} it follows that 
\(\rho^\ep( \sigma_1, \sigma_2\,;X,S,\gamma, f) =0\).
\end{proof}

\begin{example}
Although $\rho$-invariants are difficult to compute, nevertheless thanks to many authors, there is now a decent set of computations that are available. We can use these to compute end-periodic rho invariants, which we will show in a simple example. Consider $Y=S^1$ with the trivial spin structure. Then unitary characters $\sigma_1, \sigma_2$ of the fundamental group of $S^1$ can be identified with real numbers, and a computation (cf. page 82, \cite{Gilkey}) says that the rho invariant of the spin Dirac operator is, $\rho(S^1, \sigma_1, \sigma_2) = \sigma_1-\sigma_2 \mod \ZZ$. 
In particular, $\rho(S^1, \sigma_1, \sigma_2) $ can take on any real value $\mod \ZZ$. 
Let $W$ be a spin cobordism from $S^1$ to $S^1$, and $\Sigma$ be the compact spin Riemann surface (whose genus is $\ge1$) obtained as a result of gluing the two boundary components of $W$. Then $S^1$ is a codimension one submanifold of $\Sigma$ that represents a generator $a$ of $\pi_1(\Sigma)$. We can extend the characters $\sigma_1, \sigma_2$ of $a\ZZ$ to all of $\pi_1(\Sigma)$ by declaring them to be trivial on the other generators. Then by Theorem \ref{thm:RmodZ}, it follows that \(\rho^\ep(\Sigma,  \gamma,\sigma_1, \sigma_2)=\sigma_1-\sigma_2 \mod \ZZ\), can take on any real value $\mod \ZZ$, where $\gamma$ is the degree one cohomology class on $\Sigma$ which is Poincar\'e dual to $S^1$. We conclude by Theorem  \ref{thm:ev}  that the Riemann surface $\Sigma$ does not admit a PSC metric. This of course can also be proved by the Gauss-Bonnet theorem and is well known. 

The construction generalises easily to any odd dimensional spin manifold $Y$ with non-zero rho invariant $\rho(Y, \sigma_1, \sigma_2)\ne 0
\mod \ZZ$. We conclude by Theorem \ref{thm:RmodZ} that the resulting even dimensional spin manifold $X$  constructed from a spin cobordism from $Y$ to itself, has non-zero end-periodic rho invariant \(\rho^\ep(X,  \gamma,\sigma_1, \sigma_2)\ne 0 \mod \ZZ\) where $\gamma$ is the degree one cohomology class on $X$ which is Poincar\'e dual to the submanifold $Y$. In particular, such an $X$ does not admit a PSC metric. Examples of $Y$ include odd-dimensional lens spaces $L(p; \vec{q})$, where it is shown in Theorem 2.5, part (c)
\cite{Gilkey84}, that for any spin structure on $L(p; \vec{q})$, there is a representation $\sigma$ of $\pi_1(L(p; \vec{q}))$ such that  $\rho(L(p; \vec{q}), {\rm Id}, \sigma)\ne 0
\in \QQ/\ZZ$. Explicitly, for 3 dimensional lens spaces $L(p, q)$, consider the representation $\sigma : \pi_1(L(p; \vec{q})) \longrightarrow U(1)$ taking the generator $t \in \pi_1(L(p; q))$
to the unit complex number $\exp(2\pi \sqrt{-1}/p)$. Then    $\rho(L(p; q), {\rm Id}, \sigma) = -\left(\frac{d}{2p}\right) (p+1) \ne 0 \in \QQ/\ZZ$ where $d$ is a certain integer relatively prime to $48p$. Then \(\rho^\ep(X,  \gamma,{\rm Id}, \sigma)\ne 0 \in \QQ/ \ZZ\). These results confirm Theorem \ref{thm:ev} in these examples.
\end{example}

\subsection{Size of the space of components of positive scalar curvature metrics}

Hitchin \cite{Hi74} proved the first results on the size of the space of components of the space
of Riemannian metrics of positive scalar curvature metrics on a compact spin manifold, when non-empty.
This sparked much interest in the topic and results by Botvinnik-Gilkey, Piazza-Schick and many others.

We now extend Theorem \ref{thm:BG} to the even dimensional case. We would like to say something like `Given an even-dimensional manifold \(X\) with PSC having a submanifold \(Y\) of PSC Poincar{\'e} dual to a primitive one-form \(\gamma\), if \(\fM^+(Y)\) has infinitely many path components then so does \(\fM^+(X)\).' The argument would involve using a countable family of PSC metrics on \(Y\) with distinct rho invariants to find a countable such family on \(X\). There are complications however, since given an arbitrary PSC metric on \(Y\), there is not necessarily a PSC metric on \(X\) whose restriction to \(Y\) is the given metric. Because we are already assuming that there is at least one PSC metric on \(X\)  which restricts to a metric of PSC on \(Y\), there are no obstructions from topology preventing this from being the case.

\begin{definition}\label{def:psc-ad}
Let \(X\) be a compact even dimensional manifold, and \(\gamma\in H^1(X,\mathbb{Z})\) a primitive cohomology class with accompanying Poincar{\'e} dual submanifold \(Y\). Suppose that there is at least one PSC metric on \(X\) which restricts to a PSC metric on \(Y\). We say that \(X\) is \emph{psc-adaptable with respect to \(Y\)} if for \emph{every} PSC metric \(g_Y\) on \(Y\), there is a PSC metric \(g_X\) on \(X\) that is a product metric \(dt^2+g_Y\) in a tubular neighbourhood of \(Y\).
\end{definition}

\begin{figure}[h]
\includegraphics[height=1in]{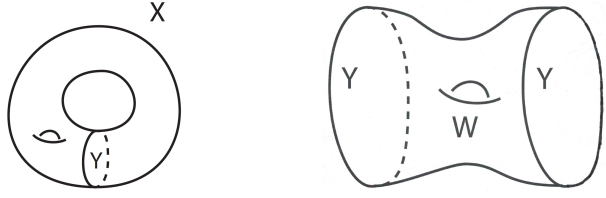}\qquad \qquad
\includegraphics[height=1in]{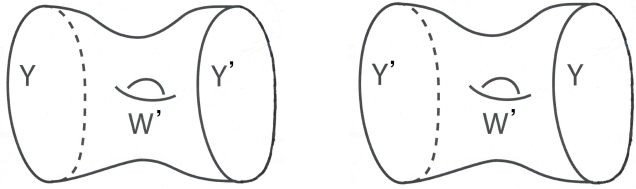}
\caption{explaining psc-adaptable}
\label{fig:pscadaptable}
\end{figure}

Some notes and comments on the notion of psc-adaptability. Let \(X\) and \(Y\) be as in the above definition, and take an arbitrary PSC metric \(g_Y\) on \(Y\). Cutting \(X\) open along \(Y\), we obtain a self cobordism \(W\) of \(Y\); see Figure \ref{fig:pscadaptable}. Under suitable assumptions on the topology of \(X\) and \(Y\), a construction of Miyazaki \cite{Miyazaki} and Rosenberg \cite{Ros} (using the theory of Gromov-Lawson \cite{GrLaw2} and Schoen-Yau \cite{SY}) enables one to {\em push} the 
psc metric on $Y$ across the bordism (pictured on the right in the figure) to get a PSC metric on $W$ restricting to metrics of PSC on each boundary component. One might then try to glue the manifold back together to obtain a PSC metric on \(X\) which restricts to the given metric \(g_Y\) on \(Y\).
The problem is that one doesn't know whether the  new psc metric on $Y$ is isotopic to the original.
% (this would be true if the general {\em concordance = isotopy conjecture} were true, cf. Botvinnik \cite{Botvinnik1, Botvinnik2, Botvinnik3}). 
Hence the concept of  psc-adaptability which hypothesizes that this is true. It is the case when the bordism is {\em symmetric} for instance. That is, starting with a bordism $W'$ from $Y$ to $Y'$, we get a bordism from $Y$ to itself by thinking of $W'$ as a bordism from $Y'$ to $Y$  and gluing to the original bordism, see Figure \ref{fig:pscadaptable2}.

\begin{figure}[h]
\includegraphics[height=1in]{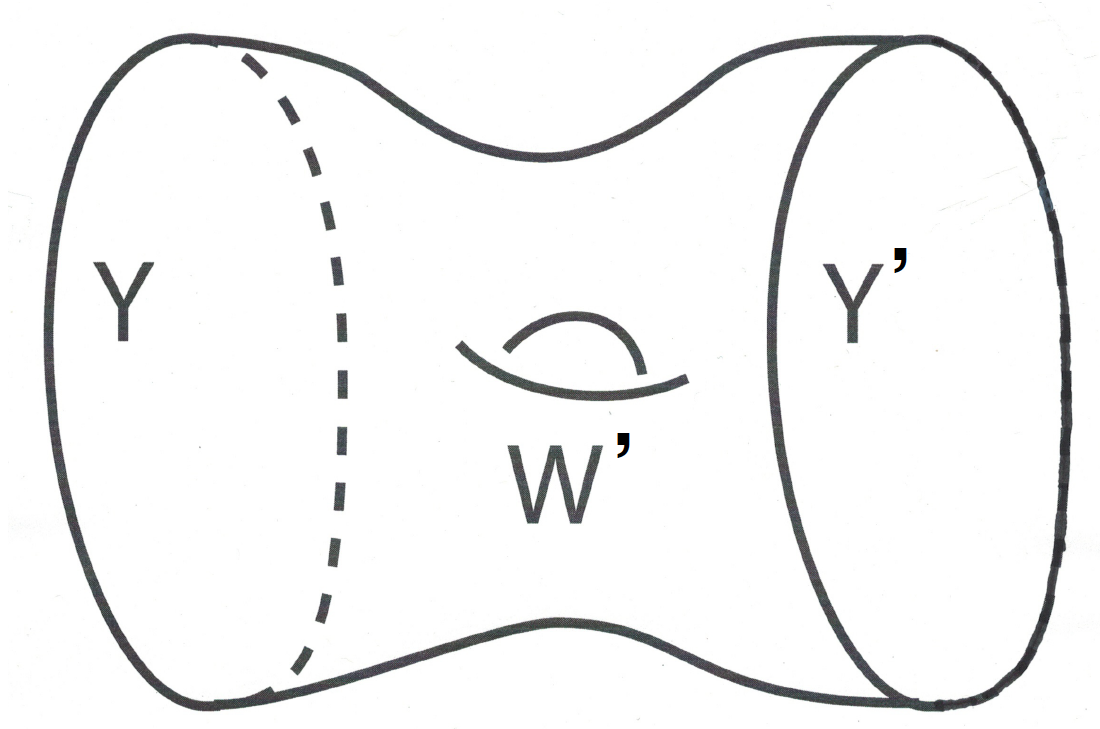}\qquad
\includegraphics[height=1in]{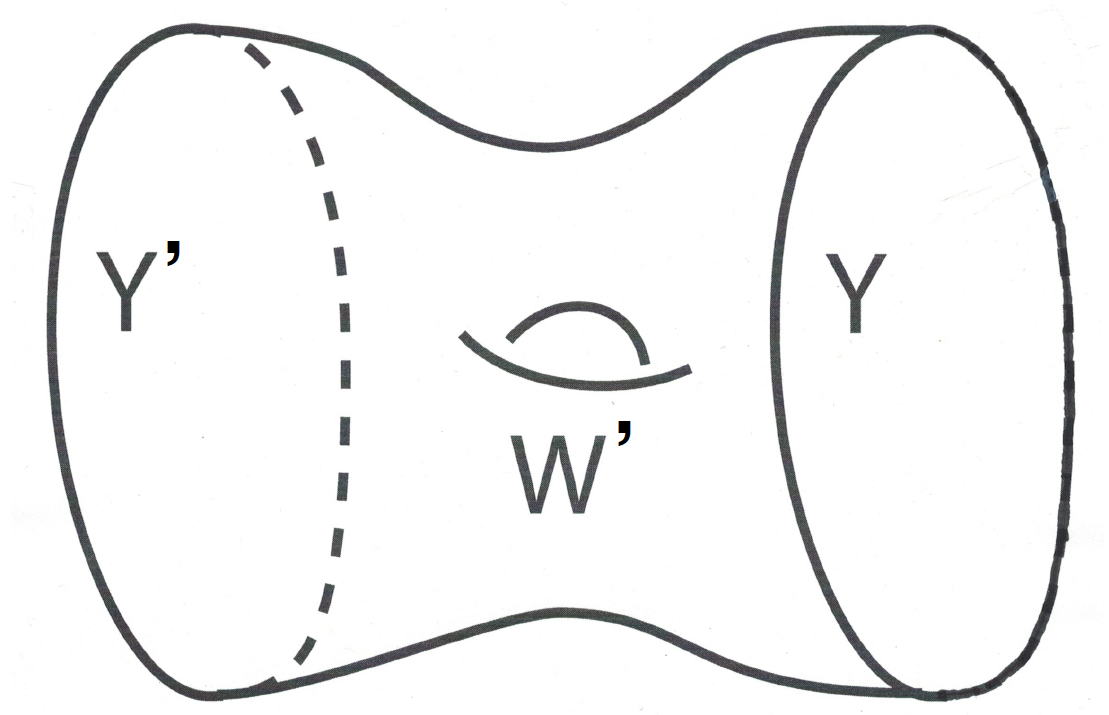}
\caption{explaining psc-adaptable}
\label{fig:pscadaptable2}
\end{figure}

Then one can use the Miyazaki-Rosenberg construction starting with the PSC metric $Y$ to get another another PSC metric on $Y'$ halfway through, and then reverse the Miyazaki-Rosenberg construction from the PSC metric on the halfway $Y'$ to get a PSC metric on $Y$ on the other end. In this case, we end up with the original PSC metric on $Y$. Since the metrics agree on either end, the bordisms can be glued together.

Mrowka, Ruberman and Saveliev also note a class of psc-adaptable manifolds -- those of the form \((S^1\times Y)\#M\) where \(Y\) and \(M\) are manifolds of positive scalar curvature, see \cite{MRS} Theorem 9.2. The end-periodic bordism groups provide a more natural framework for their proof of the following:

\begin{theorem}[Theorem 9.2, \cite{MRS}] Let \(X\) be a compact even-dimensional spin manifold of dimension \(\geq6\) admitting a metric of positive scalar curvature. Suppose there is a submanifold \(Y\subset X\) of PSC that is Poincar{\'e} dual to a primitive cohomology class \(\gamma\in H^1(X,\mathbb{Z})\), such that \(\pi=\pi_1(Y)\) is finite and non-trivial. Further assume that the classifying map \(f:Y\to B\pi\) of the universal cover extends to \(X\), and that \(X\) is psc-adaptable with respect to \(Y\). If \(r_m(\pi_1(Y))>0\), then
$\pi_0(\fM^+(X))$ is infinite, where $\fM^+(X)$
denotes the quotient of the space of positive scalar curvature metrics by the diffeomorphism group.
\end{theorem}

\begin{proof}
In the terminology of Section \ref{end-per-B}, we have an \(\Omega_m^{\eps,+}(B\pi)\)-cycle \((X,g,\sigma,\gamma,f)\), with associated \(\Omega_m^{\spin,+}(B\pi)\)-cycle \((Y,g,\sigma,f)\). Botvinnik and Gilkey \cite{BG} construct a representation \(\alpha:\pi\to U(N)\) of \(\pi\) and a countable family of metrics \(g_i\) on \(Y\) with \[\rho(\alpha,1\,;Y,g_i,\sigma,f)\neq\rho(\alpha,1\,;Y,g_j,\sigma,f)\] for \(i\neq j\), where \(1:\pi\to U(N)\) is the trivial representation. Our assumption of psc-adaptability and Theorem \ref{thm:Rmaps} imply there is an countable family of metrics \(g_i\) on \(X\) with \[\rho^\ep(\alpha,1\,;X,g_i,\sigma,\gamma,f)\neq\rho(\alpha,1\,;X,g_j,\sigma,\gamma,f)\] for \(i\neq j\). But Theorem 9.1 of MRS \cite{MRS} says that homotopic metrics of PSC on \(X\) should have the same rho invariants.
\end{proof}

%\begin{remark}
%Here $W$ is a bordism from $Y$ to $Y$ such that the metric of positive scalar curvature on $Y$ extends to a 
%metric of positive scalar curvature on $W$, inducing a metric of positive scalar curvature on $X$ obtained
%from $W$ by gluing the two boundary components of $Y$. Note that if the {\em concordance is isotopy}
%conjecture for positive scalar curvature metrics is true (cf. Botvinnik \cite{Botvinnik1, Botvinnik2, Botvinnik3}, then by results of Miyazaki  and Rosenberg \cite{Ros}, any 
%bordism  from $Y$ to $Y$ is psc-adaptable.
%\end{remark}

%\begin{theorem} Let $M$ be a compact spin manifold of dimension $4$, admitting a metric of 
%positive scalar curvature. Let $Y$ be a compact spin manifold of dimension $3$, admitting a metric of 
%positive scalar curvature such that $\pi_1(Y)$ is finite and nontrivial. Then for any positive integer $N$, 
%there is a positive integer $m_N$ such that $\pi_0(\cM^+(X\# M \# (S^2 \times S^2)))$ has at least $N$ elements, 
%where 
%$X$ is a psc-adaptable compact spin manifold of dimension $4n,$ with $Y\hookrightarrow X$ is a 
%smooth submanifold, 
%\end{theorem}

%\begin{remark}
%In proving these theorems, we define end-periodic bordism theory and generalise results of Botvinnik and Gilkey \cite{BG}
%to even dimensional manifolds,
%that are used here.
%\end{remark}

%%%%%%%%%%%%%%%%%%%%%%%
\section{Vanishing of end-periodic rho using the representation variety}\label{rep-var}
%%%%%%%%%%%%%%%%%%%%%%%%%

In this section we give a  proof of the vanishing of the end-periodic rho invariant of the twisted Dirac operator with coefficients in a flat Hermitian vector bundle on a compact even dimensional Riemannian spin manifold $X$ of positive scalar curvature using the representation variety of 
$ \pi_1(X)$ instead. 

Let $\iota\colon Y \hookrightarrow X$ be a codimension one submanifold of $X$ which is Poincar\'e dual to a generator $\gamma \in H^1(X, \ZZ)$. 
%Given a representation
%$\alpha\colon \pi_1(Y) \to U(N)$, define a representation $\tilde\alpha\colon \pi_1(X) \to U(N)$ using the commutative diagram,
%\beq\label{reps}
%\xymatrix{\pi_1(X)\ar[rr]^{\tilde\alpha}&&
%U(N)\\
%&\pi_1(Y) \ar[ul]^{\iota_*}\ar[ur]^{\alpha}} 
%\eeq

Let $\fR={\rm Hom} (\pi, U(N))$ denote the representation variety of $\pi=\pi_1(Y)$, and 
$\tilde \fR$ denote the representation variety of $\pi_1(X)$. 
We now construct a generalization of the Poincar\'e vector bundle $\cP$ over $B\pi\times \fR$. Let $E\pi \to B\pi$ be a principal $\pi$-bundle over the space $B\pi$ with contractible total space $E\pi$. Let $h \colon Y	\to B\pi$ be a continuous map classifying the universal $\pi$-covering of $Y$. We construct a tautological rank $N$ Hermitian vector bundle 
$\cP$ over $B\pi \times \fR$ as follows: consider the action of $\pi$ on $E\pi \times \fR \times \C^N$ 
given by
\begin{align*}
E\pi \times \fR  \times \C^N \times \pi &\longrightarrow E\pi \times \fR  \times \C^N\\
((q,\sigma, v), \tau) & \longrightarrow (q\tau ,\sigma, \sigma(\tau^{-1})v).
\end{align*}
Define the universal rank $N$ Hermitian vector bundle $\cP$ over $B\pi \times \fR$ to be the quotient $(E\pi \times \fR  \times \C^N)/\pi$. Then $\cP$ has the property that the restriction $\cP\big|_{B\pi \times \{\sigma\}}$ is the flat Hermitian vector bundle over $B\pi$ defined by $\sigma$. Let $I$ denote the closed unit interval $[0,1]$ 
and $\beta: I \to \fR$ be a smooth path in $\fR$ joining the unitary representation $\alpha$ to the trivial representation. 
Define $E = (f \times \beta)^*\cP \to Y \times I$ to be the Hermitian vector bundle over $Y \times I$, where $f\colon Y\to B\pi$
is the classifying map of the universal cover of $Y$. 
Let $E_t \to Y\times \{t\}$ denote the restriction of $E$ to $Y\times\{t\}$. Then $E_t$ is the flat unitary Hermitian vector bundle  
over $Y$ determined by the unitary representation $\beta(t)$ of $\pi$. Thus $E$ has a natural flat unitary connection, whose restriction
on each $E_t, \, t\in I$ is the flat unitary connection, which can be extended to a full $U(n)$-connection $\nabla^E$ on $Y\times I$, which 
amounts to giving an action of $\partial/\partial t$, or equivalently of identifying $E$ with a bundle pulled back from $Y$.
With such a choice of connection, it follows that the curvature of $E$ is a multiple of $dt$,  and so the only non-zero component of 
the Chern character form $\ch(\nabla^E)-N$ is the first Chern form $\alpha_t \wedge dt$ in dimension $2$, where $\alpha_t$ is a closed 1-form 
on $Y$, whose cohomology class $\alpha=[\alpha_t] \in H^1(Y, \R)=H^1(B\pi, \R)$ is independent of $t\in I$.

%By the Kunneth Theorem in cohomology, we have $\ch(\cP) =	\sum_i x_i \xi_i$, where $\ch(\cP)$ is the Chern character of $\cP$, for some $x_i \in H^*(B\pi, \R)$ and  $\xi_i \in H^*(\fR, \R)$.  
%Choose a closed differential form $y_i$ on $X$ such that the cohomology class $[y_i]=f^*(x_i)$, and a differential 1-form $\mu_i = f_i(t)dt$ on the interval  $[0,1]$ such that  the cohomology class $[\mu_i]=\beta^*(\xi_i)$, then the Chern character differential form $\ch(E) = \sum_i y_i\mu_i$. Note that the pullback connection makes $E$ into a Hermitian vector bundle over 
%$Y \times I$. 

\begin{theorem}[PSC and vanishing of end-periodic rho]	
Let $(X, g)$ be a compact spin manifold of even dimension, and
let $\iota\colon Y \hookrightarrow X$ be a codimension one submanifold of $X$ which is Poincar\'e dual to a primitive class 
$\gamma \in H^1(X, \ZZ)$.
Suppose that 
\begin{enumerate}
\item  $g$ is a Riemannian metric of positive scalar curvature;
\item the restriction $g\big|_Y$ is also a metric of positive scalar curvature.
\end{enumerate}

Let  $\tilde\alpha\colon \tilde\pi \to U(N)$ be a unitary representation of $\tilde\pi=\pi_1(X)$, and $\alpha\colon \pi \to U(N)$ be the 
unitary representation of $\pi=\pi_1(Y)$ defined by $\tilde\alpha\circ \iota_*$.
Assume that $\alpha$ can be connected by a smooth path $\beta : I \to \fR$ to the trivial  representation in
the representation space $\fR$.

Then  $\rho^\ep(X,S, \gamma, g;\tilde\alpha,1) =0$, 
where the flat hermitian bundle $E_{\tilde\alpha}$ is determined by $\tilde\alpha$.
\end{theorem}

\begin{proof} As observed above, the unitary connection $\nabla^E$ induced on $E$ has curvature which is a multiple of 
$dt$, so that the Chern character form $\ch(\nabla^E) = N + \alpha_t \wedge dt$, where $\alpha_t \wedge dt$ is the first Chern form of 
the connection on $E$ and 
$t$ is the variable on the interval $I$. It follows that $\ch(E) = N + \alpha\wedge dt$	where $\alpha\in H^1(Y, \R)$
is the cohomology class of $\alpha_t$.
%and 
%$\mu\in H^1(I,\R)$, as $c_1(E)$ can be represented by the trace of the curvature of a unitary connection on $E$. 
%Since $c_1(E) = (f \times \beta)^*c_1(F)$, we see that $y = f^*(x)$ and $\mu=\beta^*(\xi)$ for some
%$x\in H^1(B\pi,\R)$ and $\xi \in H^1(\fR, \R)$. Let  $E_t$ denote the flat hermitian bundle over $Y$
% determined by the representation $\gamma(t):\pi \to U(N)$.
Consider the integrand 
$\int_{Y\times I} \widehat A(Y\times I)\ch(E)$. 
Since $\widehat A(Y \times I) = \widehat A(Y)$, where $\widehat A(Y)$ is the 
A-hat characteristic class of $Y$. 
From the
discussion above
$$ \int_{Y\times I} \widehat A(Y)\ch(E)=  \int_Y \widehat A(Y) \alpha \int_I dt.$$
Since $(Y,g)$ is a spin Riemannian manifold of positive scalar curvature, it follows from Theorem 2.1 in Gromov-Lawson \cite{GrLaw} that 
$\displaystyle \int_Y \widehat A(Y)f^*(x) = 0$
for all $x\in H^1(B\pi,\R)=H^1(Y, \R)$.  

Therefore we conclude that $ \displaystyle\int_{Y\times I} \widehat A(Y)\ch(E)=0$.

Consider the manifold $Y\times I$. It can be made into an end-periodic manifold with two ends as follows. Let $W$ be the fundamental segment obtained by cutting $X$ open along $Y$, and $W_k$ be isometric copies of $W$. Then we can attach $X_1=\cup_{k\ge 0} W_k$
to one boundary component of $Y\times I$ and $X_0=\cup_{k < 0} W_k$ to the other boundary component . 
Call the resulting end-periodic manifold 
$Z_\infty$ (see the Figure \ref{fig:2ep}). It is clear that $Z_\infty$ is diffeomorphic to $\tilde X$, the cyclic Galois cover of $X$ corresponding to $\gamma$. Let $f_0 = -f$ and $f_1 = f$ for a choice of real-valued function $f$ on $Z_\infty$ such that $\gamma=[df]$.

\begin{figure}[h]
\includegraphics[height=1in]{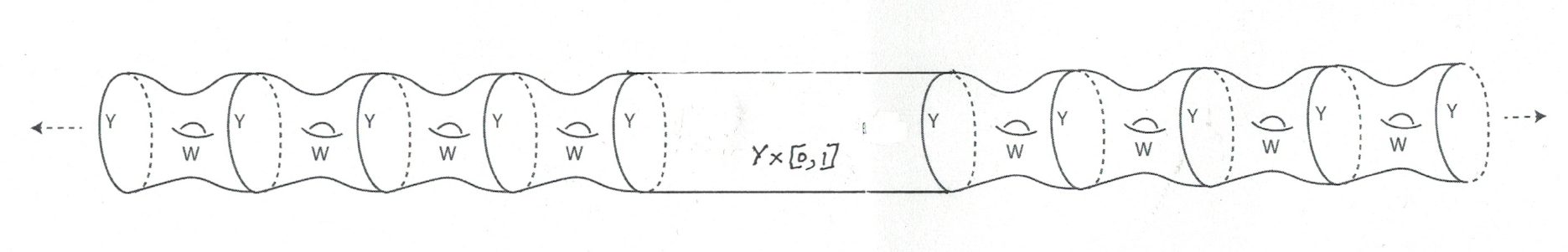}\\
\caption{End-periodic manifold with 2 ends}
\label{fig:2ep}
\end{figure}

The flat hermitian bundle $E_{\tilde\alpha}$ over $X$ induces a flat hermitian bundle $p^*(E_{\tilde\alpha})$ over $\tilde X$, where 
$p\colon \tilde X \to X$ is the projection.  The restriction of $p^*(E_{\tilde\alpha})$ to the subset $X_1$ is denoted by $E_1$. Let $E_0$ 
denote the trivial bundle over $X_0$. We use the smooth path $\gamma$ to define the bundle $E$ over $Y\times I$ which has the
property that the restriction of $\tilde E$ to the boundary components agree with $E_0$ and $E_1$, thereby defining a global vector bundle 
$\tilde E$ over $Z_\infty$.

 We can apply Theorem C in \cite{MRS} to see that
\begin{align*}
{\rm index}(D^+_{\tilde E}(Z_\infty)) &= \int_{Y\times I} \widehat A(Y\times I)\ch(E) - \int_Y \omega + \int_X df\wedge \omega 
- \frac{1}{2}(h_1 + \eta_{ep}(X, E_{\tilde\alpha}, \gamma, g)\\
& + \int_Y \omega - \int_X df\wedge \omega 
- \frac{1}{2}(h_0 - \eta_{ep}(X, E_{id}, \gamma, g)
\end{align*}
Since $g$ and $g\big|_Y$ are metrics of positive scalar curvature by hypothesis, it follows that ${\rm index}(D^+_E(Z_\infty)) =0$ 
by Lemma 8.1 in \cite{MRS} and 
that $\displaystyle\int_{Y\times I} \widehat A(Y\times I)\ch(E)=0$ by the earlier argument. Therefore $\rho^{ep}(X,S, \gamma, g;\tilde\alpha,1) =0$ as claimed.

\end{proof}

%\newpage

\end{document}